\documentclass[reqno, 12pt]{amsart}
\usepackage{amsmath,amsfonts,latexsym,graphicx,amssymb,url}
\usepackage{hyperref,pdfsync}
\usepackage{amsmath,txfonts,pifont,bbding,pxfonts,manfnt}
\usepackage[active]{srcltx}
\usepackage{wasysym,pstricks, enumerate}
\usepackage{lscape,color,subfigure,graphicx,wrapfig}
\numberwithin{equation}{section}
\usepackage{mathtools}
\usepackage{enumerate}
\usepackage{enumitem}
\usepackage{mathrsfs}
\usepackage{bbm}
\usepackage{float}

\usepackage[left=2.4cm,right=2.4cm,top=2.4cm,bottom=2.4cm]{geometry}
\usepackage{amsfonts,amssymb}
\usepackage{hyperref}
\hypersetup{
	colorlinks = true,%\
	citecolor = blue,
	filecolor=black,%
	linkcolor = blue,%[rgb]{0.00,0.50,0.00},%
	anchorcolor = red,
	pagecolor = red,
	urlcolor= [rgb]{0.65,0.0,0.0}
}

\allowdisplaybreaks
\newcommand{\diver}{\text{\rm div}}
\newcommand{\R}{{\mathbb R}} %%reals

\newcommand{\p}{\varphi}
\renewcommand{\(}{\left(}
\renewcommand{\)}{\right)}
\newtheorem{proposition}{Proposition}[section]
\newtheorem{theorem}{Theorem}[section]
\newtheorem{lemma}{Lemma}[section]
\newtheorem{corollary}{Corollary}[section]

\newtheorem{remark}{Remark}[section]

\title[Weighted Sobolev inequalities on convex cones]{Sobolev inequalities with jointly concave weights on convex cones}
\author[Zolt\'an M. Balogh, Cristian E. Guti\'errez and Alexandru Krist\'aly]{Zolt\'an M. Balogh, Cristian E. Guti\'errez and Alexandru Krist\'aly}
\thanks{\today}
\thanks{Z. M. Balogh was supported by the Swiss National Science Foundation grant 165507. C. E. Guti\'errez was partially supported by NSF grant DMS--1600578. A. Krist\'aly was supported by the National Research, Development and Innovation Fund of Hungary, financed under the K$\_$18 funding scheme, Project No.  127926.}
\address{Institute of Mathematics \\University of Bern\\Sidlerstrasse 5\\ 3012 Bern, Switzerland}
\email{zoltan.balogh@math.unibe.ch}
\address{Department of Mathematics\\Temple University\\Philadelphia, PA 19122}
\email{gutierre@temple.edu}
\address{Department of Economics\\Babe\c s-Bolyai University\\ Str.\ Teodor Mihali 58-60\\400591
 		Cluj-Napoca, Romania \& Institute of Applied Mathematics, \'Obuda
 		University, B\'ecsi \'ut 96b,  1034 
 		Budapest, Hungary}
\email{alex.kristaly@econ.ubbcluj.ro; kristaly.alexandru@nik.uni-obuda.hu}


\date{}
\begin{document}
\begin{abstract}
  Using optimal mass transport arguments, we prove weighted Sobolev inequalities of the form 
$$\(\int_E |u(x)|^q\,\omega(x) \,dx\)^{1/q}\leq  K_0\,\(\int_E |\nabla u(x)|^p\,\sigma(x)\,dx\)^{1/p},\ \ u\in C_0^\infty(\mathbb R^n),\eqno{{\rm (WSI)}}$$
where $p\geq 1$ and $q>0$ is the corresponding Sobolev critical exponent. Here $E\subseteq \mathbb R^n$ is an open convex cone, and $\omega,\sigma:E\to (0,\infty)$ are two homogeneous weights verifying a general concavity-type structural condition.  The constant $K_0= K_0(n, p, q, \omega, \sigma) >0$ is given by an explicit formula. 
Under mild regularity assumptions on the weights, we also prove  that  $K_0$  is optimal in (WSI) if and only if $\omega$ and $\sigma$ are equal up to a multiplicative factor.  Several previously known results, including the cases for monomials and radial weights, are covered by our statement. Further examples and applications to PDEs are also provided.
 
\end{abstract}
\maketitle
%hat \date{ }
\vspace{-1.2cm}

\tableofcontents

\section{Introduction}

 Driven by numerous applications to the calculus of variations and pdes, there is a rich literature of weighted Sobolev inequalities, e.g., Bakry, Gentil and Ledoux  \cite{BGL}, Kufner \cite{Kufner}, and Saloff-Coste \cite{Saloff-Coste}.  Our purpose in this paper is to prove  Sobolev inequalities for two weights of the form
\begin{equation}
\(\int_E |u(x)|^q\,\omega(x) \,dx\)^{1/q}\leq  K_0\,\(\int_E |\nabla u(x)|^p\,\sigma(x)\,dx\)^{1/p}\ \ \textrm{for all}\ u\in C_0^\infty(\R^n), \tag{WSI}\label{WSI}
\end{equation}
 with $K_0>0$ independent on $u\in C_0^\infty(\R^n)$. 
 Here $E\subseteq \mathbb R^n$ is an open convex cone, and $\omega,\sigma:E\to (0,\infty)$ are two homogeneous weights verifying some general concavity-type structural conditions to be described.  
 
There are a few ways to prove inequalities of this type when the weights $\omega$ and $\sigma$ are equal. One recent approach, based on  the ABP method, is due to Cabr\'e,  Ros-Oton and Serra, see \cite{Cabre-Ros-Oton} for monomial weights, and \cite{Cabre-Ros-Oton-Serra} for homogeneous weights. 
A second method used is based on optimal transport and was initiated by Cordero-Erausquin, Nazaret and Villani in \cite{CENV} to show the classical unweighted Sobolev inequalities. 
This second method has been further developed by Nguyen \cite{Nguyen} to deal with the case of monomial weights 
$\omega = \sigma= x_1^{\alpha_1} \ldots x_n^{\alpha_n}$ with $\alpha_i\geq 0$, $i=1,...,n$. In addition, Ciraolo, Figalli and Roncoroni \cite{CFR} recently considered the case of general $\alpha$-homogeneous weights $\omega = \sigma$ with the property that 
$\sigma^{1/\alpha}$ is concave. 

In this paper,  we continue the aforementioned line of research for two different weights $\omega$ and $\sigma$ satisfying a \textit{joint structural concavity condition} and prove  \eqref{WSI} under this assumption using optimal transport. In fact, the study of \eqref{WSI} is motivated by reaction-diffusion problems (see Cabr\'e and  Ros-Oton \cite{Cabre-Ros-Oton-2}) and Sobolev inequalities on Heisenberg groups for  axially symmetric functions (see Section \ref{sec:final comments and open questions}).  Furthermore, the cases considered in \cite{CENV}, \cite{Nguyen} and \cite{CFR} turn out to be particular cases of our results which also contain the results of Castro \cite{Castro} for possible different monomial weights, see Section \ref{sec:examples and applications}.
% and Caffarelli, Kohn and Nirenberg \cite{CKN} concerning (two possibly different) monomial or radial weights. 

We begin introducing notation and the general set up. 
Let $n\geq 2$, and let $E\subseteq \mathbb R^n$ be an open convex cone, i.e., an open convex set such that $\lambda x \in E$ for all $\lambda >0$ and $x\in E$; in particular, $0\in \overline E.$
Let $p\geq 1$ and $\omega,\sigma:E\to (0,\infty)$ be two locally integrable weights in $\overline E$, continuous in $E$, and satisfying the homogeneity conditions
\begin{equation}\label{eq:homogeneity of weights}
\omega(\lambda\, x)=\lambda^\tau\,\omega(x),\qquad 
\sigma(\lambda\,x)=\lambda^\alpha\,\sigma(x)\ \ \textrm{for all}\ \lambda>0, x\in E,
\end{equation}
 where the  parameters $\tau,\alpha\in\mathbb R$ verify 
\begin{equation}\label{eq:p-range}
1\leq  p<\alpha+n\leq\tau+p+n,
\end{equation}
and
\begin{equation}\label{eq:p-range-bis}
\alpha \geq \(1-\frac{p}{n}\)\tau.
\end{equation}
Clearly, the local integrability of $\omega$ and $\sigma$ implies that $\tau+n>0$ and $\alpha+n>0,$ respectively. Moreover, \eqref{eq:p-range} implies that $\alpha > -n +1$. 
We remark that both integrals in \eqref{WSI} are considered only on $E$ and the functions $u$ involved need not vanish on $\partial E.$
By scaling, \eqref{WSI} implies the dimensional balance condition  
\begin{equation}\label{eq:dimensional balance}
\dfrac{\tau+n}{q}=\dfrac{\alpha+n}{p}-1.
\end{equation}
The choice of the precise parameter range given by \eqref{eq:p-range} and \eqref{eq:p-range-bis} is not arbitrary; indeed, these ranges are \textit{necessary} for the validity of \eqref{WSI} as it is shown in Section \ref{sec:necessity of the conditions for wsi}. 
From \eqref{eq:dimensional balance} and (\ref{eq:p-range}), we immediately obtain that 
$$q=\frac{p(\tau+n)}{\alpha+n-p}\geq p.$$ 
An important quantity, called fractional dimension $n_a$, is given by  
\begin{equation} \label{eq:fracdim}
\dfrac{1}{n_a}=\dfrac{1}{p}-\dfrac{1}{q}.
\end{equation}
From \eqref{eq:dimensional balance}, the inequality (\ref{eq:p-range-bis}) is equivalent to 
$$n_a\geq n.$$
It may happen that $n_a=+\infty$ which is equivalent to $p=q$, i.e., to $\alpha=p+\tau$.
As usual, denote $p'=\dfrac{p}{p-1}$ for $p>1$, and $p'=+\infty$ when $p=1$.

In addition to the homogeneity assumption \eqref{eq:homogeneity of weights} and necessary conditions \eqref{eq:p-range}-\eqref{eq:dimensional balance}, 
%In spite of the fact that relations \eqref{eq:p-range},  \eqref{eq:p-range-bis} and \eqref{eq:dimensional balance} are necessarily, they are not sufficient to prove the validity of \eqref{WSI}, see again \S \ref{sect-2-1} (EXAMPLE NEEDED!!!).
%In order to prove \eqref{WSI}, in addition to (\ref{eq:homogeneity of weights}), (\ref{eq:p-range}), (\ref{eq:p-range-bis}) and the dimensional balance condition \eqref{eq:dimensional balance},  
we assume 
that the weights $\omega,\sigma:E\to (0,\infty)$ are differentiable a.e. in $E$ and satisfy either one of the following  joint structural concavity conditions.

\setlist[enumerate,1]{start=0}
\begin{enumerate}[label=\textbf{C-\arabic*}]
	\item\label{item:C_0}: If $n_a>n,$ then there exists a constant $C_0>0$ such that 
	\begin{equation} \label{c-0-condition}
	\(\(\dfrac{\sigma(y)}{\sigma(x)}\)^{1/p} \,\(\dfrac{\omega(x)}{\omega(y)}\)^{1/q}\)^{n_a/(n_a-n)}
	\leq
	C_0\,\(\dfrac{1}{p'}\dfrac{\nabla \omega(x)}{\omega(x)}+\dfrac{1}{p}\dfrac{\nabla \sigma(x)}{\sigma(x)}\)\cdot y
	\end{equation} 
	for a.e. $x\in E$ and for all $y\in E$. 
	\item\label{item:C_1}: If  $n_a=n$, then $\sup_{x\in E}\dfrac{\omega(x)^{1/q}}{\sigma(x)^{1/p}}=:C_1\in (0,\infty)$, and
	
	\begin{equation}\label{c-1-feltetel}
	0
	\leq
	\(\dfrac{1}{p'}\dfrac{\nabla \omega(x)}{\omega(x)}+\dfrac{1}{p}\dfrac{\nabla \sigma(x)}{\sigma(x)}\)\cdot y
	\end{equation}
	for a.e. $x\in E$ and for all $y\in E$.
\end{enumerate}

\noindent We notice that whenever $\omega=\sigma$ is a homogeneous weight of degree $\alpha>0$ and $C_0=\frac{1}{\alpha}$, relation \eqref{c-0-condition} in \ref{item:C_0} turns to be equivalent to the concavity of  $\sigma^{1/\alpha}$, see \cite[Lemma 5.1]{Cabre-Ros-Oton-Serra}. Even more, Proposition \ref{prop-1} reveals an unexpected  rigidity connection between condition  \ref{item:C_0} and the concavity of the weights $\omega$ and $\sigma $ in a limiting case.

Our main results are that under either one of these assumptions \eqref{WSI} holds. Our first main result is then as follows.

	\begin{theorem}\label{main-theorem-1}
	Let $p>1$,  $E\subseteq \mathbb R^n$ be an open convex cone  and weights $\omega,\sigma:E\to (0,\infty)$ satisfying relations \eqref{eq:homogeneity of weights}-\eqref{eq:dimensional balance}, continuous in $E$ and differentiable a.e. in $E$. Then we have
	\begin{itemize}
		\item[{\rm (i)}]   if %$n_a>n$ and 
		condition {\rm \ref{item:C_0}}  holds  for some $C_0>0$, then \eqref{WSI} holds
		with
		\[
		K_0=\max\left\{C_0\,\(1-\dfrac{n}{n_a}\),\dfrac{1}{n_a}\right\}\,\,q\,\(\dfrac{1}{p'}+\dfrac{1}{q}\)\inf_{\int_E v(y)\,dy=1, v\in C_0^\infty(\mathbb R^n), v\geq 0}
		\dfrac{\(\int_E v(y)\,|y|^{p'}\,dy\)^\frac{1}{p'}}{\int_E v(y)^{1-\frac{1}{n_a}}\omega\(y\)^{-\frac{1}{q}}\,
			\sigma\(y\)^{\frac{1}{p}}\,dy};
		\]

		\item[{\rm(ii)}] if condition {\rm \ref{item:C_1}} holds for some $C_1>0$, then \eqref{WSI} holds with 
		\[
		K_0=\frac{C_1}{n}\,q\,\(\dfrac{1}{p'}+\dfrac{1}{q}\)\,\inf_{\int_E v(y)\,dy=1, v\in C_0^\infty(\mathbb R^n), v\geq 0}
		\dfrac{\(\int_E v(y)\,|y|^{p'}\,dy\)^\frac{1}{p'}}{\int_E v(y)^{1-\frac{1}{n}}\,dy}.
		\]
	\end{itemize}	
\end{theorem}

%	We  notice that when $n_a=+\infty$  (equivalently to $\alpha=p+\tau$), the constant
%	 $\max \left\{ C_0\,\(1-\dfrac{n}{n_a}\),\dfrac{1}{n_a} \right\}$ in Theorem \ref{main-theorem-1}/(i) should be understood as $C_0.$

The proof of this theorem is based on optimal transport arguments \`a la  Cordero-Erausquin,  Nazaret and Villani \cite{CENV}.
% by using our key conditions \ref{item:C_0-0} and \ref{item:C_1-0}. 
The statement of the theorem is general enough to cover several well-known results and flexible enough to apply to new cases as well. A well-known Sobolev inequality for radial weights of the form $\omega(x) = |x|^{\tau}$ and $\sigma(x) = |x|^{\alpha}$ (see  Caffarelli, Kohn and Nirenberg  \cite{CKN}) follows as a corollary of this theorem. 
 Considering equal weights  $\omega=\sigma$ in Theorem \ref{main-theorem-1}/(i) we recover the isotropic weighted Sobolev inequality in \cite[Appendix A]{CFR} and \cite{Nguyen} when $\omega=\sigma=w$ is a monomial weight. 
When $\omega$ and $\sigma$ are monomial weights not necessarily equal, Theorem \ref{main-theorem-1} contains also the main result of Castro \cite{Castro}, providing in addition an explicit Sobolev constant in \eqref{WSI}. 
Moreover, our setting allows that some parameters $\tau_i\in \mathbb R$ 
in the monomial $\omega(x_1,...,x_n)=x_1^{\tau_1}\cdots x_n^{\tau_n}$ can take negative values, 
which is an unexpected phenomenon that does not appear in the papers \cite{Cabre-Ros-Oton, CFR, Nguyen}.

When $p=1$, with a proof similar to that of Theorem \ref{main-theorem-1}, we obtain isoperimetric-type inequalities for two weights. In this case,  we have $\frac{1}{n_a}+\frac{1}{q}=1$ and $\frac{1}{p'}=0$, and both conditions \ref{item:C_0} and \ref{item:C_1} are understood with these values; 
see \eqref{degenerate-C-0} and the end of the proof of Lemma \ref{lm:condition for two weights}. 
 For further use, let $B:=\{x\in \mathbb R^n:|x|\leq 1\}$. 
Our second main result is then the following.

\begin{theorem}\label{main-theorem-2} 
	Let $p=1$,  $E\subseteq \mathbb R^n$ be an open convex cone  and weights $\omega,\sigma:E\to (0,\infty)$ satisfying relations \eqref{eq:homogeneity of weights}-\eqref{eq:dimensional balance}, continuous in $E$ and differentiable a.e. in $E$. Then we have
	\begin{itemize}
		\item[{\rm (i)}]   if %$n_a>n$ and 
		condition {\rm \ref{item:C_0}}  holds  for some $C_0>0$, then \eqref{WSI} holds
		with
		\[
		K_0=\max\left\{C_0\,\(1-\dfrac{n}{n_a}\),\dfrac{1}{n_a}\right\}
		\dfrac{\(\int_{B\cap E} \omega(y)\,dy\)^{1-\frac{1}{n_a}}}{\int_{B\cap E} 
			\sigma\(y\)\,dy};
		\]
	%	where $\tilde C_0$ is the constant in \eqref{eq:divergence estimate two weights};
		\item[{\rm (ii)}] if condition {\rm \ref{item:C_1}} holds for some $C_1>0$, then \eqref{WSI} holds with 
		\[
		K_0=\frac{C_1}{n}\,\dfrac{\(\int_{B\cap E} \omega(y)\,dy\)^{1-\frac{1}{n}}}{\int_{B\cap E} 
			\omega\(y\)^{1-\frac{1}{n}}\,dy}.
		\]
	\end{itemize}
	Moreover, inequality \eqref{WSI} extends to functions with $\sigma$-bounded variation on $E$.	
\end{theorem}

This statement covers the main results in \cite{Cabre-Ros-Oton-Serra} on weighted isoperimetric inequalities when $\omega = \sigma$. To be more precise, let us introduce a few definitions to conclude from Theorem \ref{main-theorem-2} isoperimetric inequalities.
 A function $f:\mathbb R^n\to\mathbb R$ has $\sigma$-bounded variation on $E$ if 
 $$V_\sigma(f,E)=\sup\left\{\int_{E} f(x){\rm div}(\sigma(x) X(x))dx:X\in C_0^1(E,\mathbb R^n),|X(x)|\leq 1,\forall x\in E\right\}<+\infty.$$
Let $BV_\sigma(\mathbb R^n)$ be the set of these functions. It is clear that $\dot W_\sigma^{1,1}(\mathbb R^n)\subset BV_\sigma(\mathbb R^n)$ and for every  $u\in \dot W_\sigma^{1,1}(\mathbb R^n)$, we have $$\int_{E} |\nabla u(x)|\sigma(x)dx=V_\sigma(u,E).$$
Here for each  $p\geq 1$,  
$\dot W_\sigma^{1,p}(\mathbb R^n)$ denotes the set of  all measurable functions $u:\mathbb R^n\to \mathbb R$ such that the level sets $\{x\in E:|u(x)|>s\},$ $s>0,$ have finite $\sigma$-measure and  $|\nabla u|\big|_E\in L_\sigma^p(E)$, the space of functions that are $p$-th integrable with respect to $\sigma$ in $E$.
%; usually, $L_\sigma^p(E)$  stands for the space of measurable functions $v:E\to \mathbb R$ such that $$\int_E|v(x)|^p\sigma(x)dx<+\infty.$$

A measurable set $\Omega\subset \mathbb R^n$ has $\sigma$-bounded variation on $E$ if $\mathbbm{1}_\Omega\in BV_\sigma(\mathbb R^n)$, and its weighted perimeter with respect to the convex cone $E$ is given by
$$ P_\sigma(\Omega,E)=V_\sigma(\mathbbm{1}_\Omega,E).$$
  The conclusions of Theorem \ref{main-theorem-2} can be then reformulated in terms of  weighted isoperimetric inequalities, i.e., for any set $\Omega\subset \mathbb R^n$ having $\sigma$-bounded variation on $E$, one has   
  \begin{equation}\label{isoperimetric-ineq}
  K_0^{-1}\(\int_{\Omega\cap E}\omega(x)dx\)^{1-\frac{1}{n_a}}\leq P_\sigma(\Omega,E),
  \end{equation}
where $K_0>0$ is the constant given by Theorem \ref{main-theorem-2}.
When $\omega=\sigma$,  \eqref{isoperimetric-ineq} 
is the sharp weighted isoperimetric inequality of \cite{Cabre-Ros-Oton-Serra}, and \cite{Nguyen} in the monomial case. Moreover, for different monomial weights we recover from \eqref{isoperimetric-ineq} the results of Abreu and Fernandes \cite{AF}.

The next question considered is to describe the \textit{equality} cases in Theorems \ref{main-theorem-1} and \ref{main-theorem-2}. As expected, the candidates for extremal functions belong to $\dot W_\sigma^{1,p}(\mathbb R^n)$ rather than to $C_0^\infty(\mathbb R^n)$. Therefore, we may assume that  \eqref{WSI} is extended to functions in $\dot W_\sigma^{1,p}(\mathbb R^n)$. The equality cases in Theorems \ref{main-theorem-1} and \ref{main-theorem-2} are described in the following result.

\begin{theorem}\label{main-theorem-1-equality}
	Let $p\geq 1$, $E\subseteq \mathbb R^n$ be an open convex cone and weights $\omega,\sigma:E\to (0,\infty)$ satisfying relations \eqref{eq:homogeneity of weights}-\eqref{eq:dimensional balance}, continuous in $E$, differentiable a.e. in $E$, and one of them locally Lipschitz in $E.$
	 Then we have:
	\begin{itemize}
		\item[{\rm (i)}]   if %$n_a>n$ and 
		condition {\rm \ref{item:C_0}}  holds  for some $C_0>0$ and $n_a<+\infty$, then there exist nonzero extremal functions in \eqref{WSI} $($with the constant $K_0$ in Theorem {\rm \ref{main-theorem-1}/(i) or Theorem \ref{main-theorem-2}/(i)}$)$ if and only if $\omega$ and $\sigma$ are equal up to a multiplicative factor, $\sigma^\frac{1}{\alpha}$ is concave, and $C_0=\frac{1}{n_a-n}$; 
		%		where $h(x)=\omega\(x\)^{-1/q}\,
		%		\sigma\(x\)^{1/p},$ and 
		%		$$\tilde C_0=\max\left\{C_0\,\(1-\dfrac{n}{n_a}\),\dfrac{1}{n_a}\right\};$$
		\item[{\rm (ii)}]   if %$n_a>n$ and 
		condition {\rm \ref{item:C_0}}  holds and $n_a=+\infty$, there are no extremal functions in \eqref{WSI};
		\item[{\rm(iii)}] if condition {\rm \ref{item:C_1}} holds for some $C_1>0$, then there exist nonzero extremal functions in \eqref{WSI} $($with the  constant $K_0$ in Theorem {\rm \ref{main-theorem-1}/(ii) or Theorem \ref{main-theorem-2}/(ii)}$)$ if and only if both weights are constant, i.e., $\omega\equiv c_\omega>0$ and $\sigma\equiv c_\sigma>0$ with $c_\omega^\frac{1}{q}=C_1c_\sigma^\frac{1}{p}$. 		
	\end{itemize}
	
\end{theorem}

Theorem \ref{main-theorem-1-equality} follows by a careful analysis of the equality cases in the proof of Theorems \ref{main-theorem-1} and \ref{main-theorem-2}. 
Besides the regularity properties of the optimal transport map -- similar to those in \cite{CENV} (see also \cite{Nguyen} when the weights are two equal monomials) -- the main novelty in our argument is a rigidity phenomenon showing up from conditions  \ref{item:C_0} and \ref{item:C_1} which implies that the weights $\omega$ and $\sigma$ are \textit{equal} up to a multiplicative factor.  For a technical reason, in order to establish  Theorem \ref{main-theorem-1-equality}, our argument requires further regularity on the weights with respect to Theorems \ref{main-theorem-1} and \ref{main-theorem-2}, 
that is, one of them is assumed to be locally Lipschitz.
On one hand, Theorem \ref{main-theorem-1-equality} shows in a certain sense the limits of our approach. 
In particular, no characterization can be provided for the equality cases in axially symmetric Sobolev inequalities on the Heisenberg group $\mathbb H^1$, since in that case $\omega/\sigma\neq$constant (see Section \ref{sec:final comments and open questions}). 
On the other hand, Theorem \ref{main-theorem-1-equality} shows that the results from \cite{Cabre-Ros-Oton-Serra}, \cite{CFR} and \cite{Nguyen} are optimal in the sense that the only reasonable scenario to obtain sharp \eqref{WSI} inequalities with the constants given above is when the two weights are constant multiples of each other. 
The difference between the cases $p>1$ and $p=1$ in Theorem \ref{main-theorem-1-equality}/(i) and (iii) appears in the shape of the extremal functions. In the former case it is Talenti-type radial function (independently on the weight), while in the latter case is the indicator function of $B\cap E$.

 We complete this introduction summarizing the organization of the paper. 
%In Section \ref{sec:statement of results}, we state the main results in the paper and comment on its relation to previous results. 
In Section \ref{sec:proof of main theorems} we prove Theorems \ref{main-theorem-1} and  \ref{main-theorem-2}. Section \ref{sec:discussion of conditions} begins with a discussion concerning a concavity rigidity arising from condition \ref{item:C_0}, and then we provide the proof of Theorem \ref{main-theorem-1-equality}.  
In Section \ref{sec:examples and applications} we give various examples and applications of our results. 
In particular, examples of pairs of weights $(\omega,\sigma)$ satisfying conditions \ref{item:C_0} and \ref{item:C_1} are given in Section \ref{subsection5-1} showing that several known results are simple corollaries of Theorems \ref{main-theorem-1} and \ref{main-theorem-2}. In Section \ref{eigenvalue-1} we provide some applications by estimating the spectral gap in a weighted eigenvalue problem and discuss the existence of nontrivial weak solution for a weighted PDE. 
Finally, in Section \ref{sec:necessity of the conditions for wsi}, we show that \eqref{eq:p-range}--\eqref{eq:dimensional balance} are necessary conditions for the validity of \eqref{WSI}, and next in Section \ref{sec:final comments and open questions}
we establish the relation between \eqref{WSI} and Sobolev inequalities in the Heisenberg group. We finish the paper with final comments and open questions.\\

%\medskip

{\bf Acknowledgements:} We thank Alessio Figalli and Xiao Zhong for several motivating conversations on this subject.
C. E. G. and A. K. wish to thank Zolt\'an Balogh for his kind invitation to the University of Bern during part of the Fall of 2019, and they are extremely grateful for the wonderful hospitality and support.
%C. E. G. is grateful to spend part his sabbatical leave from Temple University at the University of Bern to collaborate in this project. 

  \section{Proof of Theorems \ref{main-theorem-1} and \ref{main-theorem-2} }\label{sec:proof of main theorems}
 
We start this section with  some preliminary remarks on conditions  
\ref{item:C_0} and \ref{item:C_1}. 
Let us notice, that from Euler's theorem for homogeneous functions, one has 
$\nabla \omega (x) \cdot x = \tau \omega(x) $ and $\nabla \sigma (x) \cdot x = \alpha \sigma(x) $ for a.e. $x \in E$. 
Picking $y=x\in E$ in \ref{item:C_0} yields $1\leq C_0\(\frac{\tau}{p'}+\frac{\alpha}{p}\),$
implying that if \ref{item:C_0} holds, then at least one of the parameters $\tau$ or $\alpha$ must be strictly positive. 
Clearly, \ref{item:C_1} holds for constant weights.

\begin{remark}\rm \label{remark-1}
	(i) 
	%By homogeneity, it is enough to consider condition \ref{item:C_1} on the set $E\cap \mathbb S^{n-1}$. 
Using \eqref{eq:dimensional balance} and \eqref{eq:fracdim}, condition \ref{item:C_0} can be written  in terms of $\alpha$ and $\tau$ as follows
	\begin{equation} \label{eq:explicit-c-0}  \left( \left(\frac{\sigma(y)}{\sigma(x)}\right)^{\tau +n} \left( \frac{\omega(x)}{\omega(y)}\right)^{\alpha +n -p}\right )^{\frac{1}{n(\alpha-\tau) + p \tau}} \leq 
	C_0\,\(\dfrac{1}{p'}\dfrac{\nabla \omega(x)}{\omega(x)}+\dfrac{1}{p}\dfrac{\nabla \sigma(x)}{\sigma(x)}\)\cdot y,
	\end{equation}
for a.e. $x\in E$ and all $y\in E$.	
	
	(ii) When $n_a=+\infty$ (i.e., $p=q$, which is equivalent to $\alpha=p+\tau$), from (i), it is easy to see, that condition \ref{item:C_0} takes the form 
	\begin{eqnarray}\label{c-0-reduced-0}
	\(\dfrac{\sigma(y)}{\sigma(x)}\dfrac{\omega(x)}{\omega(y)}\)^{1/p} 
	\leq
	C_0\,\(\dfrac{1}{p'}\dfrac{\nabla \omega(x)}{\omega(x)}+\dfrac{1}{p}\dfrac{\nabla \sigma(x)}{\sigma(x)}\)\cdot y\ 	\ \ {\rm for\ a.e.}\ x\in E\ {\rm  and\ all}\ y\in E.
	\end{eqnarray}

(iii) When $n_a\to n$ in condition \ref{item:C_0}, the only reasonable relation we obtain is precisely (\ref{c-1-feltetel}) in condition \ref{item:C_1}. Indeed, if we fix $x,y\in E$ such that $\dfrac{\omega(x)^{1/q}}{\sigma(x)^{1/p}}<\dfrac{\omega(y)^{1/q}}{\sigma(y)^{1/p}}$ then the left hand side of \eqref{c-0-condition} tends to 0 whenever $n_a\to n$.

	(iv) When $n_a=n$,  \eqref{eq:dimensional balance} implies 
	$\dfrac{\tau}{q}=\dfrac{\alpha}{p}$, and so by \eqref{eq:homogeneity of weights} the function $\dfrac{\omega^{1/q}}{\sigma^{1/p}}$ is homogeneous of degree zero. Thus, the constant $C_1$ in condition \ref{item:C_1} equals
$$C_1 := \sup_{x\in E \cap \mathbb S^{n-1} }\frac{\omega(x)^{1/q}}{\sigma(x)^{1/p}} < \infty.$$
In spite of the fact that $\dfrac{\omega^{1/q}}{\sigma^{1/p}}$ is homogeneous of degree zero, the last condition is not automatically satisfied; indeed, the function $(x_1, x_2) \mapsto \frac{x_1}{x_2}$ is 0-homogeneous in $E=(0,\infty)^2$ but it certainly blows up when $x_2 \to 0^+$.
\end{remark}

{
	
%	The next part in blue is Cristian's version; I did not modify it: 	

}
%Just for us: note that this is an important condition that is NOT automatically satisfied by all 0-homogeneous functions. For example the function $(x_1, x_2) \to \frac{x_1}{x_2}$ is 0-homogeneous in the upper right quarter but it certainly blows up as $x_2 \to 0$. 
 
 \subsection{Weighted divergence type inequalities}
 	\setcounter{equation}{0}
 The proof of Theorems \ref{main-theorem-1} and 	\ref{main-theorem-2} are based on a pointwise divergence type inequality stated in the following lemma. 
 	Let us recall that if $\phi:\mathbb R^n\to \mathbb R$ is a convex function, $D^2_A\phi$ denotes its Hessian in the sense of Alexandrov, i.e., the absolutely continuous part of the distributional Hessian of $\phi$, see e.g. Villani \cite{Villani}. In the same sense, let $\Delta_A \phi={\rm tr}D^2_A\phi$ be the Laplacian and for $f\in C^1(\mathbb R^n)$, let $\diver_A(f\nabla \phi)=\nabla f\cdot \nabla \phi+f\Delta_A \phi.$

 	\begin{lemma}\label{lm:condition for two weights}
 		%We assume that $\nabla \omega(x)\cdot y$ and $\nabla \sigma(x)\cdot y$ are non negative for all $x,y\in E$\footnote{By Euler's theorem if $\omega,\sigma$ satisfy the homogeneity condition \eqref{eq:homogeneity of weights}, then this clearly holds.}. 
 		Let $\omega,\sigma:E\to (0,\infty)$ be weights satisfying
 		\eqref{eq:homogeneity of weights}-\eqref{eq:dimensional balance}, continuous in $E$ and differentiable a.e. in $E$. Let $\phi:\mathbb R^n\to \mathbb R$ be a convex function such that $\nabla \phi(E)\subseteq E$.

Then we have 

\begin{itemize}
	\item[{\rm (i)}] If %$n_a>n$ and 
	{\rm \ref{item:C_0}} holds with $C_0>0$, then for a.e. $x\in E$ one has
		$$\omega(x)^{1-\frac{1}{n_a}}\,\omega\(\nabla \phi(x)\)^{-1/q}\,
	\sigma\(\nabla \phi(x)\)^{1/p}
	\,
	\(\det D^2_A\phi(x)\)^{1/n_a}
	\leq 
	\tilde C_0\,
	\diver_A\(\omega(x)^{1/p'}\sigma(x)^{1/p}\,\nabla \phi\),$$
with 
\begin{equation}\label{eq:divergence estimate two weights}
	\tilde C_0=\max\left\{C_0\,\(1-\dfrac{n}{n_a}\),\dfrac{1}{n_a}\right\}.
	\end{equation}
	
	\item[{\rm (ii)}] If %$n_a=n$ and 
	{\rm \ref{item:C_1}} holds with $C_1>0$, then 
	$$	
	\omega(x)^{1-\frac{1}{n_a}}
	\(\det D^2_A\phi(x)\)^{1/n_a}
	\leq 
	\frac{C_1}{n_a}\,
	\diver_A\(\omega(x)^{1/p'}\sigma(x)^{1/p}\,\nabla \phi\)\ \textrm{for a.e.}\ x\in E.
	$$

\end{itemize}
 	\end{lemma}

\begin{proof} 
 		
% Let us notice first, that since $\phi$ is convex, it is differentiable almost everywhere. Moreover, by Alexandrov's theorem $D^2_A\phi(x)$ exists almost everywhere. It goes without saying that in our proof we shall only consider points $x$ where these objects are well defined. 
 
 	Let us begin proving (i). We divide the proof into several cases. 
 	
 	\textit{Case 1}: $p>1$ and $n_a<+\infty$.
	%We assume that % $n_a>n$ and \ref{item:C_0} holds with $C_0>0$. 
	Since $\nabla \phi(E)\subseteq E$, $\omega\(\nabla \phi(x)\)$ and $\sigma\(\nabla \phi(x)\)$ are well-defined for a.e. $x\in E.$	Therefore, for a.e. $x\in E$, we have 
 		\begin{align*}
 		&\omega(x)^{1-\frac{1}{n_a}}\,\omega\(\nabla \phi(x)\)^{-1/q}\,
 		\sigma\(\nabla \phi(x)\)^{1/p}
 		\,
 		\(\det D^2_A\phi(x)\)^{1/n_a}\\
 		&\leq
 		\omega(x)^{1-\frac{1}{n_a}}\,\omega\(\nabla \phi(x)\)^{-1/q}\,
 		\sigma\(\nabla \phi(x)\)^{1/p}
 		\,
 		\(\dfrac{\Delta_A \phi(x)}{n}\)^{n/n_a}\qquad \text{(from the AM-GM inequality)}\\
 		&=\omega(x)^{1-\frac{1}{n_a}}\,
 		\(\dfrac{\omega\(\nabla \phi(x)\)^{-1/q}\,
 			\sigma\(\nabla \phi(x)\)^{1/p}}{\omega(x)^{-n/qn_a}\,\sigma(x)^{n/pn_a}}\)
 		\,
 		\(\dfrac{\Delta_A \phi(x)}{n}\,\omega(x)^{-1/q}\,\sigma(x)^{1/p}\)^{n/n_a}\\
 		&=
 		\omega(x)^{1-\frac{1}{n_a}}\,
 		\(\(\dfrac{\omega\(\nabla \phi(x)\)^{-1/q}\,
 			\sigma\(\nabla \phi(x)\)^{1/p}}{\omega(x)^{-n/qn_a}\,\sigma(x)^{n/pn_a}}\)^{n_a/(n_a-n)}\)^{1-\frac{n}{n_a}}
 		\,
 		\(\dfrac{\Delta_A \phi(x)}{n}\,\omega(x)^{-1/q}\,\sigma(x)^{1/p}\)^{n/n_a}\\
 		&\leq
 		\omega(x)^{1-\frac{1}{n_a}}\,
 		\(
 		\(1-\frac{n}{n_a}\)
 		\(\dfrac{\omega\(\nabla \phi(x)\)^{-1/q}\,\sigma\(\nabla \phi(x)\)^{1/p}}{\omega(x)^{-n/qn_a}\,\sigma(x)^{n/pn_a}}\)^{n_a/(n_a-n)}
 		+
 		\dfrac{1}{n_a}
 		\,\omega(x)^{-1/q}\,\sigma(x)^{1/p}\,\Delta_A \phi(x)\)\\
 		&\leq
 		\omega(x)^{1-\frac{1}{n_a}}\,
 		\(1-\frac{n}{n_a}\)
 		\(\dfrac{\(C_0\,\(\dfrac{1}{p'}\dfrac{\nabla \omega(x)}{\omega(x)}+\dfrac{1}{p}\dfrac{\nabla \sigma(x)}{\sigma(x)}\)\cdot \nabla \phi(x)\)^{(n_a-n)/n_a}\,\,\omega(x)^{-1/q}\sigma(x)^{1/p} }{\omega(x)^{-n/qn_a}\,\sigma(x)^{n/pn_a}}\)^{n_a/(n_a-n)}\\
 		&\qquad \qquad +
 		\dfrac{1}{n_a}
 		\omega(x)^{1-\frac{1}{n_a}}\,\omega(x)^{-1/q}\,\sigma(x)^{1/p}\,\Delta_A \phi(x)\qquad \qquad\text{(from \ref{item:C_0})}
 		\\
 		&=
 		\omega(x)^{1-\frac{1}{n_a}}\,
 		\(1-\frac{n}{n_a}\)
 		\dfrac{C_0\,\(\dfrac{1}{p'}\dfrac{\nabla \omega(x)}{\omega(x)}+\dfrac{1}{p}\dfrac{\nabla \sigma(x)}{\sigma(x)}\)\cdot \nabla \phi(x)}{\omega(x)^{1/q}\,\sigma(x)^{-1/p}}
 		+
 		\dfrac{1}{n_a}
 		\omega(x)^{1-\frac{1}{n_a}}\,\omega(x)^{-1/q}\,\sigma(x)^{1/p}\,\Delta_A \phi(x)\\
 		&=
 		\omega(x)^{1/p'}\,\sigma(x)^{1/p}\,\(1-\frac{n}{n_a}\)
 		\,C_0\,\(\dfrac{1}{p'}\dfrac{\nabla \omega(x)}{\omega(x)}+\dfrac{1}{p}\dfrac{\nabla \sigma(x)}{\sigma(x)}\)\cdot \nabla \phi(x)
 		+
 		\dfrac{1}{n_a}
 		\omega(x)^{1/p'}\,\sigma(x)^{1/p}\,\Delta_A \phi(x)\\
 		&
 		\leq
 		\max\left\{C_0\,\(1-\frac{n}{n_a}\),\dfrac{1}{n_a}\right\}
 		\(\omega(x)^{1/p'}\,\sigma(x)^{1/p}
 		\,\(\dfrac{1}{p'}\dfrac{\nabla \omega(x)}{\omega(x)}+\dfrac{1}{p}\dfrac{\nabla \sigma(x)}{\sigma(x)}\)\cdot \nabla \phi(x)
 		+
 		\omega(x)^{1/p'}\,\sigma(x)^{1/p}\,\Delta_A \phi(x)\)\\
 		&=
 		\max\left\{C_0\,\(1-\frac{n}{n_a}\),\dfrac{1}{n_a}\right\}
 		\,
 		\diver_A\(\omega(x)^{1/p'}\sigma(x)^{1/p}\,\nabla \phi\),
 		\end{align*}
 		which proves (i) whenever $p>1$. 
In the above estimates we used that both terms $\Delta_A \phi(x)$ and $\left( \frac{1}{p'} \frac{\nabla \omega (x)}{\omega(x)} + \frac{1}{p} \frac{\nabla \sigma(x)}{\sigma(x)} \right) \cdot \nabla \phi(x) $ are nonnegative.  	
	
  \textit{Case 2}: $p=1$ and $n_a<+\infty$.	Then $\frac{1}{n_a}+\frac{1}{q}=1$ and $\frac{1}{p'}=\frac{p-1}{p}=0$; accordingly,   condition \ref{item:C_0} takes the form 
 		\begin{equation}\label{degenerate-C-0}
 		\(\dfrac{\sigma(y)}{\sigma(x)} \,\(\dfrac{\omega(x)}{\omega(y)}\)^{(n_a-1)/n_a}\)^{n_a/(n_a-n)}
 		\leq
 		C_0\dfrac{\nabla \sigma(x)}{\sigma(x)}\cdot y\ \ \ \textrm{for all}\ x,y\in E.
 		\end{equation}
 		A similar argument as before gives 
 		\begin{equation*}	 
 		\omega(x)^{1-\frac{1}{n_a}}\,\omega\(\nabla \phi(x)\)^{-1/q}\,
 		\sigma\(\nabla \phi(x)\)
 		\,
 		\(\det D^2_A\phi(x)\)^{1/n_a}
 		\leq 
 		\tilde C_0\,
 		\diver_A\(\sigma(x)\,\nabla \phi(x)\)\ \textrm{for a.e.}\ x\in E,
 		\end{equation*}
 		which is the desired inequality.
 		
	\textit{Case 3}: $p>1$ and $n_a=+\infty$. Since  $n_a=+\infty$,	we have $q=p$. Thus, by \eqref{c-0-reduced-0} and $\Delta_A \phi(x)\geq 0$ for a.e. $x\in E,$ it turns out that 
 		\begin{eqnarray*}
 		\omega(x)\,\omega\(\nabla \phi(x)\)^{-1/p}\,
 		\sigma\(\nabla \phi(x)\)^{1/p}
 		&\leq&
 		C_0\omega(x)^{\frac{1}{p'}}\,\sigma\(x\)^{1/p}\,
 		\(\dfrac{1}{p'}\dfrac{\nabla \omega(x)}{\omega(x)}+\dfrac{1}{p}\dfrac{\nabla \sigma(x)}{\sigma(x)}\)\cdot \nabla\phi(x)\\
 		&\leq&
 		C_0
 		\,
 		\diver_A\(\omega(x)^{1/p'}\sigma(x)^{1/p}\,\nabla \phi(x)\),
 		\end{eqnarray*}
 		which is the required inequality with  $\tilde C_0=C_0$.

 		\textit{Case 4}: $p=1$ and $n_a=+\infty$. Since in this case $p=q=1$, 
 		 condition \ref{item:C_0} reduces to
 		\begin{equation}\label{degenerate-C-0-0}
 		\dfrac{\sigma(y)}{\sigma(x)} \dfrac{\omega(x)}{\omega(y)}
 		\leq
 		C_0\dfrac{\nabla \sigma(x)}{\sigma(x)}\cdot y\ \ \ \textrm{for a.e. $x\in E$ and all}\ y\in E.
 		\end{equation}
 		Therefore, by \eqref{degenerate-C-0-0} and $\Delta_A \phi(x)\geq 0$ for a.e. $x\in E,$ one has
 		\begin{equation*}		
 		\omega(x)\,\omega\(\nabla \phi(x)\)^{-1}\, \sigma\(\nabla \phi(x)\)	\leq 
 		C_0\,
 		\nabla\sigma(x)\cdot\nabla \phi(x)
 		\leq 
 		 C_0\,
 		\diver_A\(\sigma(x)\,\nabla \phi(x)\)\ \textrm{for a.e.}\ x\in E,
 	\end{equation*}
 		concluding the proof of (i).

 		To show(ii), we divide the proof into two parts.
 		
\textit{Case 1}: $p>1$ and $n_a=n$.	 Since $n_a=n$, one has $\frac{1}{p} - \frac{1}{q} = \frac{1}{n}$.  Moreover, by the definition of $C_1>0$ in condition
		\ref{item:C_1} it follows that		
 		\begin{equation}\label{kozbeeso}
 			\omega(x)^{1-\frac{1}{n}}\leq C_1\omega(x)^{1/p'}\sigma(x)^{1/p},\ \ x\in E.
 		\end{equation}
 		 Then for a.e. $x\in E$ one has
 		\begin{eqnarray*}
 		\omega(x)^{1-\frac{1}{n_a}}
 		\(\det D^2_A\phi(x)\)^{1/n_a}
 		&=& \omega(x)^{1-\frac{1}{n}}
 		\(\det D^2_A\phi(x)\)^{1/n}\\ 
 		&\leq & 
 		\omega(x)^{1-\frac{1}{n}}
 		\dfrac{\Delta_A \phi(x)}{n}\text{\hspace{1.5cm}(from the AM-GM inequality)}\\ 
 		&\leq & 
 		\frac{C_1}{n}\omega(x)^{1/p'}\sigma(x)^{1/p}
 	{\Delta_A \phi(x)}\text{\hspace{1.5cm}}\\ 
 		&\leq & 
 		\frac{C_1}{n}\(\(\dfrac{1}{p'}\dfrac{\nabla \omega(x)}{\omega(x)}+\dfrac{1}{p}\dfrac{\nabla \sigma(x)}{\sigma(x)}\)\cdot \nabla \phi(x)+\omega(x)^{1/p'}\sigma(x)^{1/p}
 		{\Delta_A \phi(x)}\)\\&&\hspace{5cm} \text{\hspace{0.5cm}(from $\nabla \phi(E)\subseteq E$ and \ref{item:C_1})}
 		\\&=&
 		\frac{C_1}{n}\diver_A\(\omega(x)^{1/p'}\sigma(x)^{1/p}\,\nabla \phi(x)\),
 		\end{eqnarray*}
 		which concludes the proof whenever $p>1$. 
 		
 		\textit{Case 2}: $p=1$ and $n_a=n$. Since $p=1$, one has $\frac{1}{p'} = 0$, and  condition \ref{item:C_1} reads as $\sup_{x\in E}\dfrac{\omega(x)^{1/q}}{\sigma(x)}=C_1\in (0,\infty)$ and  $0
 		\leq
 		\nabla \sigma(x)\cdot y$ {for all} $ x,y\in E.$ In particular, since $\frac{1}{q}=1-\frac{1}{n}$, then  $\omega(x)^{1-\frac{1}{n}}\leq C_1\sigma(x)$ for every $x\in E.$
 		A similar argument as in the previous case  provides the inequality 
 		\begin{equation*}		
 			\omega(x)^{1-\frac{1}{n_a}}
 		\(\det D^2_A\phi(x)\)^{1/n_a}\leq\frac{C_1}{n}\diver_A\(\sigma(x)\,\nabla \phi(x)\)\  \ \textrm{for a.e.}\ x\in E,
 	    \end{equation*}
 		which concludes the proof of the lemma. 
 	\end{proof}

\subsection{Proof of Theorem \ref{main-theorem-1}}
 
 	From Lemma \ref{lm:condition for two weights} we can now give the proof of the desired weighted Sobolev inequalities on convex cones.
% 	% \end{equation}
% 	% for each $u\in C_0^\infty (E)$.
%  	
% \textbf{INNEN  KIVETTEM a tetelt!!!}
 
 %{\it Proof of Theorem \ref{main-theorem-1}.}
  Let $u\in C_0^\infty(\mathbb R^n)$ be fixed.
 		If $\mathcal L^n({\rm supp}(u)\cap E)=0$, we have nothing to prove; hereafter, $\mathcal L^n$ stands for the $n$-dimensional Lebesgue measure. Thus, we may assume that $\mathcal L^n({\rm supp}(u)\cap E)>0$ and to simplify the notation, let $U={\rm supp}(u)$. We may assume that $u$ is nonnegative and by scaling
 		$$\int_E u(x)^q\,\omega(x)\,dx=1.$$
 		We also fix $v\in C_0^\infty(\mathbb R^n)$ a nonnegative function  satisfying 
 		\[
 		 \int_E v(y)\,dy=1.
 		\]
 		Consider the probability measures in $E$, $\mu=u^q\,\omega\,dx$ and $\nu=v\,dy$, and
 		let $T$ be the optimal map with respect to the quadratic cost such that $T_\sharp \mu=\nu$. By Brenier's theorem there is $\phi$ convex in $\mathbb R^n$ such that $T=\nabla \phi$ and $\nabla \phi(E)\subseteq {\rm supp}\nu \subseteq E$.
 		This is equivalent to the following Monge-Amp\`ere equation
 		\begin{equation}\label{eq:MA equation two weights-1-0}
 		u^q(x)\,\omega(x)=v\(\nabla \phi(x)\)\,\det D^2_A\phi(x)\ \ \textrm{for a.e.}\ x\in U\cap E.
 		\end{equation}
 	
 	Proof of (i). 
 		Raising \eqref{eq:MA equation two weights-1-0} to the power $1-\frac{1}{n_a}$ and rewriting the resulting equation yields
 		\begin{equation}\label{eq:MA equation two weights-2-0}
 		v^{1-\frac{1}{n_a}}\(\nabla \phi(x)\)\,h(\nabla \phi(x))\,\det D^2_A\phi(x)
		=u^{q\(1-\frac{1}{n_a}\)}(x)\,\omega^{1-\frac{1}{n_a}}(x)\,h(\nabla \phi(x))[\det D^2_A\phi(x)]^\frac{1}{n_a},
 		\end{equation}
 		where $h(x)=\omega\(x\)^{-1/q}\,
 		\sigma\(x\)^{1/p}.$
% 		If $\beta:=1/n_a=1/p-1/q$, then raising the last identity to the power $1-\beta$ and multiplying the resulting identity by $h\(\nabla \phi(x)\)$, clearly yields
% 		\[
% 		u(x)^{q(1-\beta)}\,\omega(x)^{1-\beta}\,h\(\nabla \phi(x)\)\,\(\det D^2\phi(x)\)^\beta=v\(\nabla \phi(x)\)^{1-\beta}\,h\(\nabla \phi(x)\)\,\det D^2\phi(x).
% 		\]
 		Integrating this identity over $U\cap E$, changing variables on the LHS, and using Lemma \ref{lm:condition for two weights}/(i) on the RHS, yields
 		\begin{align*}
 		\int_E v(y)^{1-\frac{1}{n_a}}\,h(y)\,dy
 		&\leq
 		\tilde C_0\,\int_{U\cap E} u(x)^{q\(1-\frac{1}{n_a}\)}\,\diver_A\(\sigma(x)^{1/p}\omega(x)^{1/p'}\,\nabla \phi\)\,dx:=\tilde C_0\,I.
 		\end{align*}
% 		$$
% =-C_1\,\int_E \nabla\(u(x)^{q(1-\beta)}\)\cdot \nabla \phi(x)\,\sigma(x)^{1/p}\omega(x)^{1/p'}\,dx\quad\text{from the divergence theorem}
% 		=-C_1\,q\,(1-\beta)\,
% 		\int_E \nabla u(x)\cdot \nabla \phi(x)\,u(x)^{q(1-\beta)-1}\,\sigma(x)^{1/p}\omega(x)^{1/p'}\,dx
% 		\leq
% 		C_1\,q\,(1-\beta)\,
% 		\(\int_E |\nabla u(x)|^p\,\sigma(x)\,dx\)^{1/p}
% 		\(\int_E u(x)^{(q(1-\beta)-1)p'}\,\omega(x)\,|\nabla \phi(x)|^{p'}\,dx\)^{1/p'},$$
 		Since $\Delta_A\phi\leq \Delta_{\mathcal D'}\phi$, where $\Delta_{\mathcal D'}$ is the distributional Laplacian, integrating by parts one gets 
 	\begin{eqnarray}\label{integr-part}
 	I&\leq &\int_{U\cap E} u^{q\(1-\frac{1}{n_a}\)}(x)\,{\rm div}_{\mathcal D'}(\omega(x)^{\frac{1}{p'}}\sigma(x)^{\frac{1}{p}}\nabla \phi(x))dx\nonumber\\&= &
 		\int_{\partial {(U\cap E)}} u^{q\(1-\frac{1}{n_a}\)}(x)\,\omega(x)^{\frac{1}{p'}}\sigma(x)^{\frac{1}{p}}\nabla \phi(x) \cdot {\bf n}(x)\,ds(x)\nonumber \\&&-q\(1-\frac{1}{n_a}\)\int_{U\cap E} u^{q\(1-\frac{1}{n_a}\)-1}(x)\,\omega(x)^{\frac{1}{p'}}\sigma(x)^{\frac{1}{p}}\nabla \phi(x)\cdot \nabla u(x)\,dx,
 	\end{eqnarray}
 		where ${\bf n}(x)$ is the outer normal vector at $x\in \partial (U\cap E).$ 
		%Let $x\in \partial (U\cap E)$ be fixed.  
		Since $E$ is a convex cone,  $y\cdot {\bf n}(x)\leq 0$ for each $y\in \bar E$ and $x\in \partial E$. 
		In particular, $\nabla \phi(x)\cdot {\bf n}(x)\leq 0$ for each $x\in \partial E$, since $\nabla \phi(\overline E)\subseteq \overline E$. 
		On the other hand,		
		$\partial (U\cap E)\subset \partial U\cup \partial E.$ So
		we obtain that the integrand in the boundary integral is nonpositive for $x\in \partial E$ and is zero for $x\in \partial U$ since $q\(1-\frac{1}{n_a}\)>0$.
%		
%		we have the following two cases: 
% 		
% 		\begin{itemize}
% 			\item if $x\in \partial E$, by the convexity of $E$ and the fact that $\nabla \phi(\overline E)\subseteq \overline E$, it follows that ${\bf n}(x)$ forms an obtuse angle with $\nabla \phi(x)$, i.e. $\nabla \phi(x) \cdot {\bf n}(x)\leq 0$,  			
% 			thus the  integrand in the first integral  on the RHS is non-positive; 
% 			\item if $x\in \partial U=\partial ({\rm supp}(u))$ then $u(x)=0$ and since  $q(1-\frac{1}{n_a})>0$, the same integrand vanishes this time.  
% 		\end{itemize}
 		Therefore, the boundary integral in \eqref{integr-part} can be dropped and by H\"older's inequality
		it follows that 
 		\begin{equation*}
 		I\leq q\(1-\frac{1}{n_a}\)\(\int_{E} u^{q}(x)\omega(x)|\nabla \phi(x)|^{p'}dx\)^\frac{1}{p'}\(\int_{E} |\nabla u(x)|^p\sigma (x)dx\)^\frac{1}{p},
 		\end{equation*}
		since $\(q\(1-\frac{1}{n_a}\)-1\)p'=q$.
 	Using once again the Monge-Amp\`ere equation \eqref{eq:MA equation two weights-1-0} yields
 		\begin{equation*}
 		\int_E u(x)^{q}\,\omega(x)\,|\nabla \phi(x)|^{p'}\,dx
 		=
 		\int_E v\(\nabla \phi(x)\)\,|\nabla \phi(x)|^{p'}\,\det D^2_A\phi(x)\,dx
		=
 		\int_E v(y)\,|y|^{p'}\,dy.
 		\end{equation*}
 	Therefore, the above estimates give
 		$$\int_E v(y)^{1-\frac{1}{n_a}}\,h(y)\,dy
		\leq \tilde C_0\,q\(1-\frac{1}{n_a}\)\(\int_E v(y)\,|y|^{p'}\,dy\)^\frac{1}{p'}\(\int_{E} |\nabla u(x)|^p\sigma (x)\,dx\)^\frac{1}{p}.$$
 		which completes the proof of (i).
 		
 		Proof of (ii). Since \ref{item:C_1} holds, one has $n_a=n$. From \eqref{eq:MA equation two weights-1-0}, we have
 		\begin{equation*}\label{eq:MA equation two weights-2-1}
 		v^{1-\frac{1}{n_a}}\(\nabla \phi(x)\)\,\det D^2_A\phi(x)=u^{q\(1-\frac{1}{n_a}\)}(x)\,\omega^{1-\frac{1}{n_a}}(x)[\det D^2_A\phi(x)]^\frac{1}{n_a},\ x\in E.
 		\end{equation*}
 		Integrating the last equation and using Lemma \ref{lm:condition for two weights}/(ii) gives
 			\begin{equation}\label{isop-kell}
 		\int_E v(y)^{1-\frac{1}{n_a}}\,dy
 		\leq
 		\frac{C_1}{n_a}\,\int_{U\cap E} u(x)^{q\(1-\frac{1}{n_a}\)}\,\diver_A\(\sigma(x)^{1/p}\omega(x)^{1/p'}\,\nabla \phi\)\,dx.
 		\end{equation}
 		We now proceed as in case (i), obtaining that 
 		$$	\int_E v(y)^{1-\frac{1}{n_a}}\,dy
 		\leq
 		\frac{C_1}{n_a}q\(1-\frac{1}{n_a}\)\(\int_E v(y)\,|y|^{p'}\,dy\)^\frac{1}{p'}\(\int_{E} |\nabla u(x)|^p\sigma (x)\)^\frac{1}{p},$$
 		which completes the proof of the theorem. 	
 	\hfill $\square$

% 	In the sequel, we consider the case when $p=1$; in particular,  we have $\frac{1}{n_a}+\frac{1}{q}=1$ and $\frac{1}{p'}=0$. Accordingly,  both conditions \ref{item:C_0} and \ref{item:C_1} are understood in this sense, see \eqref{degenerate-C-0} and the end of the proof of Lemma \ref{lm:condition for two weights}. For further use, let $B:=B_1=\{x\in \mathbb R^n:|x|\leq 1\}$. In the sequel, we state the following isoperimetric-type  inequality for different weights, namely, 
% 		\begin{equation}\label{eq:isoperimetric}
% 	\(\int_E |u(x)|^\frac{n_a}{n_a-1}\,\omega(x) \,dx\)^{1-\frac{1}{n_a}}\leq K_0\,\int_E |\nabla u(x)|\,\sigma(x)\,dx\ \ \textrm{for every}\ u\in C_0^\infty (\mathbb R^n).
% 	\end{equation}

 %INNEN KIVEVE
 \subsection{Proof of Theorem \ref{main-theorem-2}}
 
% {\it Proof of Theorem \ref{main-theorem-2}.} 	
%It is similar to the proof of Theorem \ref{main-theorem-1}.
Let us start with an arbitrarily fixed nonnegative function $u\in C_0^\infty(\mathbb R^n)$ with the property $\int_E u(x) ^{\frac{n_a}{n_a-1}} \omega(x) dx = 1$, and $v(y):=\frac{\omega(y)}{\int_{B\cap E}\omega(y)dy}\mathbbm{1}_{B\cap E}(y)$. Let us consider the optimal transport map $T = \nabla \phi$ such that   $T_{\sharp} \mu = \nu $ for $\mu = u^{\frac{n_a}{n_a-1}} \omega dx $ and $\nu = v dx$. We may repeat the arguments from Theorem \ref{main-theorem-1} with suitable modifications. 
 	
 Proof of (i).	If  \ref{item:C_0} holds, then since $\nabla \phi(x)\in {\rm supp}v=B\cap E$ for a.e. $x\in U\cap E$, we can use  Lemma \ref{lm:condition for two weights}/(i) for $p=1$. In this case we notice that $1-\frac{1}{n_a} = \frac{1}{q}$.  The divergence theorem and $\nabla \phi(x)\in B\cap E$ for a.e. $x\in U\cap E$ imply 
 	\begin{eqnarray*}
 		\int_{B\cap E} v(y)^{1-\frac{1}{n_a}}\,\omega\(y\)^{-\frac{1}{q}}\,
 		\sigma\(y\)\,dy
 		&\leq&
 		\tilde C_0\,\int_{U\cap E} u(x)^{q\(1-\frac{1}{n_a}\)}\,\diver_A\(\sigma(x)\,\nabla \phi\)\,dx\\
 		&=&\tilde C_0\,\int_{U\cap E} u(x)\,\diver_A\(\sigma(x)\,\nabla \phi\)\,dx\\
 		&\leq&\tilde C_0\left(\int_{\partial {(U\cap E)}} u(x)\sigma(x)\nabla \phi(x) \cdot {\bf n}(x)\,ds(x)\right.\\&&\left.-\int_{U\cap E} \sigma(x)\nabla u(x)\cdot \nabla \phi(x)\,dx \right)\\
 		&\leq&\tilde C_0\int_{U\cap E} \sigma(x)|\nabla u(x)||\nabla \phi(x)| \,dx \\
 		&\leq&\tilde C_0\int_{ E} |\nabla u(x)|\sigma(x) \,dx. 
 	\end{eqnarray*}
 	Using again the relation  $1-\frac{1}{n_a} = \frac{1}{q}$, we obtain 
	
	$$\int_{B\cap E} v(y)^{1-\frac{1}{n_a}}\,\omega\(y\)^{-\frac{1}{q}}\,
 	\sigma\(y\)\,dy=\frac{\int_{B\cap E} 
 	\sigma\(y\)\,dy}{\(\int_{B\cap E} \omega(y)\,dy\)^{1-\frac{1}{n_a}}}.$$
 %	and the claim follows. 
 	
 	Proof of (ii).  Suppose, that condition \ref{item:C_1} holds for some $C_1 >0$.  In this case, instead of (\ref{isop-kell}), we use Lemma \ref{lm:condition for two weights}/(ii) for $p=1$. We conclude 	\begin{align*}
 	\int_{B\cap E} v(y)^{1-\frac{1}{n_a}}\,dy
 	&\leq
 	\frac{C_1}{n_a}\,\int_{U\cap E} u(x)^{q\(1-\frac{1}{n_a}\)}\,\diver_A\(\sigma(x)\,\nabla \phi\)\,dx=\frac{C_1}{n_a}\,\int_{U\cap E} u(x)\,\diver_A\(\sigma(x)\,\nabla \phi\)\,dx.
 	\end{align*}
 	Proceeding as before yields 
 	$$	\int_E v(y)^{1-\frac{1}{n_a}}\,dy
 	\leq
 	\frac{C_1}{n_a}\int_{E} |\nabla u(x)|\sigma (x)dx,$$
 	which concludes the proof. 
 	
 	Clearly, both  (i) and (ii) can be extended to functions with $\sigma$-bounded variation on $E$.	
 \hfill $\square$

 \begin{remark} \rm Theorems \ref{main-theorem-1} and  \ref{main-theorem-2} can be formulated in the \textit{anisotropic} setting as well, by considering any norm instead of the usual Euclidean one.
 The only technical difference is the use of H\"older's inequality for the norm and its polar transform, see e.g.\ \cite{CFR, CENV}. 
 	When $\omega=\sigma=1$, the weights are homogeneous of degree zero and one has $n_a = n$. Choosing $E=\mathbb R^n$,   condition \ref{item:C_1} trivially holds with constant $C_1=1$. Thus  Theorems \ref{main-theorem-1}/(ii) and \ref{main-theorem-2}/(ii) yield the well-known sharp Sobolev inequality  ($p>1$) and sharp isoperimetric inequality $(p=1)$, respectively, in Del Pino and  Dolbeault \cite{delPinoDolbeault-1, delPinoDolbeault-2} and Cordero-Erausquin,  Nazaret and Villani \cite[Theorems 2 and 3]{CENV}.
 \end{remark}

\section{Discussion of the equality cases: proof of Theorem \ref{main-theorem-1-equality}}\label{sec:discussion of conditions}

In this section we are going to prove Theorem \ref{main-theorem-1-equality}, that is, to  identify the equality cases in Theorems \ref{main-theorem-1} and \ref{main-theorem-2}. 
As we already pointed out after the statement of Theorem \ref{main-theorem-2}, we may extend \eqref{WSI} from $C_0^\infty(\mathbb R^n)$ to  $\dot W_\sigma^{1,p}(\mathbb R^n)$, that is larger space in order to search for a suitable candidate as an extremal function. To do this extension, a careful approximation argument is needed which is similar to the one carried out in \cite[Lemma 7]{CENV} for the unweighted case, and that was adapted to equal monomial weights in \cite{Nguyen}. In fact, the idea to do this is to extend the integration by parts formula \eqref{integr-part} to functions $u$ in $\dot W_\sigma^{1,p}(\mathbb R^n)$, a technical issue discussed in detail in \cite{CENV, Nguyen}. Since the same technique can be adapted also to our setting, we thus omit the details.  

In order to prove Theorem \ref{main-theorem-1-equality} we shall need some preliminary results. First, we have the following characterization of concavity. 

\begin{lemma}\label{concave-characterization}
	Let $E\subseteq \mathbb R^n$ be an open convex set and $h:E\to \mathbb R$ be a continuous function which is a.e. differentiable in $E$. Then the following statements are equivalent: 
	\begin{itemize}
		\item[{\rm (a)}] $h$ is concave in $E;$
		\item[{\rm (b)}] for a.e. $x\in E$ and all $y\in E$ one has 
		$h(y) - h(x)\leq \nabla h (x) \cdot (y-x).$
	\end{itemize}
\end{lemma}

 \begin{proof} Although  standard, we provide the proof since we did not find it in the literature.  The implication "(a)$\implies$(b)" is trivial. For "(b)$\implies$(a)", let $E_0\subset E$ be the set where $h$ is differentiable; clearly,  $\mathcal L^n(E\setminus E_0)=0$. Let 	
 	 $x_0,y_0\in E$, $0<t<1$, and $z_0=(1-t)x_0+ty_0$.
 	If $z_0\in E_0$, then by our assumption, we have that 
 	$h(x_0)-h(z_0)\leq \nabla h(z_0)\cdot (x_0-z_0)$ and $h(y_0)-h(z_0)\leq \nabla h(z_0)\cdot (y_0-z_0)$.
 	Multiplying the first inequality by $(1-t)$, the second by $t$, and adding them up yields
 	$(1-t)\,h(x_0)+t\,h(y_0)-h(z_0)\leq 0$.
 	On the other hand, if $z_0\notin E_0$, pick a sequence $z_k\in E_0$ such that $z_k\to z_0$.
 	Since $E$ is open, we can pick sequences $x_k,y_k\in E$ such that $x_k\to x_0$, $y_k\to y_0$, with $z_k=(1-t)x_k+ty_k$. In particular, we have that
 	$h(x_k)-h(z_k)\leq \nabla h(z_k)\cdot (x_k-z_k)$ and $h(y_k)-h(z_k)\leq \nabla h(z_k)\cdot (y_k-z_k)$.
 	Multiplying the latter inequality by $t$ and the former by $(1-t)$ yields $(1-t)\,h(x_k)+t\,h(y_k)-h(z_k)\leq 0$.
 	Since $h$ is continuous, letting $k\to \infty$ we obtain the concavity of $h$.
 	\end{proof} 

%In the case when  $n_a>n$ and $\omega=\sigma$,   condition $(C_0)$ holds with $C_0=\frac{1}{\alpha}$ reduces to the concavity of $\sigma^{\frac{1}{\alpha}}$, covering in this way the setting  considered by Ciraolo,  Figalli and Roncoroni \cite{CFR}.

We are ready to prove a rigidity result based on the validity of condition \ref{item:C_0}. 

\begin{proposition}\label{prop-1} 
%	Assume $n_a>n$ and $\omega=\sigma$, and thus $\tau=\alpha$. 
Let $E\subseteq \mathbb R^n$ be an open convex cone and weights $\omega,\sigma:E\to (0,\infty)$ satisfying relation  \eqref{eq:homogeneity of weights} with $\alpha>0$, $\tau\in \mathbb R$, continuous in $E$, differentiable a.e. in $E$. Assume in addition that at least one of the weights $\omega$ or $\sigma$ is locally Lipschitz in $E.$ If $n_a<+\infty$, 
we have
	\begin{enumerate}
		\item[{\rm (i)}] 
		if condition {\rm \ref{item:C_0}} holds with $C_0>0$ and $\tau\leq \alpha$, then $C_0\geq \frac{1}{n_a-n};$
%		 if {\rm \ref{item:C_0}} holds for $\omega,\sigma:E\to (0,\infty)$ with $\omega=c\sigma$ for some $c>0$, then $C_0\geq 1/\alpha;$
		\item[{\rm (ii)}] the following statements are equivalent: 
		\begin{itemize}
			\item[{\rm (a)}] condition {\rm \ref{item:C_0}} holds for $C_0= \frac{1}{n_a-n}$ and $\tau\leq \alpha$;  
			\item[{\rm (b)}] $\omega=c\sigma$ for some $c>0$  $($thus $\alpha=\tau$$)$ and $\sigma^{1/\alpha}$ is concave in $E$.
		\end{itemize}
		
	\end{enumerate} 
\end{proposition} 

\begin{proof} 
	
	(i) %Assume that $\omega = c \sigma$ for some $c>0$. 
	From Euler's theorem for homogeneous functions,  
$\nabla \omega (x) \cdot x = \tau \omega(x) $ and $\nabla \sigma (x) \cdot x = \alpha \sigma(x) $ for all a.e. $x \in E$. Picking $y=x\in E$ in \ref{item:C_0} yields 
	  $1\leq C_0\(\dfrac{\tau}{p'}+\dfrac{\alpha}{p}\)$ .
Using \eqref{eq:dimensional balance} and \eqref{eq:fracdim} we get that
$n_a = \dfrac{p\,(\tau +n)}{\tau - \alpha +p},$
and 
\begin{equation*}
n_a - n = \frac{p\, \tau + n(\alpha - \tau)}{\tau - \alpha +p} \geq \tau + \frac{n}{p}(\alpha -\tau) 
\geq \tau + \frac{1}{p}(\alpha - \tau) = \frac{\tau}{p'}+\frac{\alpha}{p},
\end{equation*}  	  
where in the last estimate we used the assumption $\tau\leq \alpha$.  
The lower estimate for $C_0$ then follows.
%This estimate combined with $1\leq C_0\(\frac{\tau}{p'}+\frac{\alpha}{p}\)$ proves the first claim. 
	
	(ii) "(b)$\implies$(a)" On one hand, by Lemma \ref{concave-characterization}, we notice that the concavity of $\sigma^{1/\alpha}$  in $E$ implies that 	$$ {\sigma(y)}^{1/\alpha}-{\sigma(x)}^{1/\alpha}
	\leq
	\nabla \sigma^{1/\alpha}(x)\cdot (y-x)\ \ \ \textrm{for a.e. $x\in E$ and all}\ y\in E.$$
	By the 1-homogeneity of $\sigma^{1/\alpha}$  and  Euler's theorem, it turns out that ${\sigma(x)}^{1/\alpha}=\nabla \sigma^{1/\alpha}(x)\cdot x$ for a.e. $x\in E$, thus the last inequality is equivalent to 
	\begin{equation}\label{concave-0}
	 {\sigma(y)}^{1/\alpha}
	\leq
	\nabla \sigma^{1/\alpha}(x)\cdot y=\frac{1}{\alpha}\sigma(x)^{1/\alpha-1}\nabla \sigma(x)\cdot y\ \ \ \textrm{for a.e. $x\in E$ and all}\ y\in E.
	\end{equation}
	On the other hand, since by assumption $\omega=c\sigma$ (for some $c>0$), one has $\tau=\alpha$ and $n_a=n+\alpha$. 
	Now using \eqref{eq:explicit-c-0} we see that 
	condition \ref{item:C_0}  means
$$
	\sigma(x)\(\dfrac{\sigma(y)}{\sigma(x)}\)^{1/\alpha}
	\leq
	C_0\nabla \sigma(x)\cdot y\ \ \ \textrm{for a.e. $x\in E$ and all}\ y\in E.
$$
On account of (\ref{concave-0}), condition \ref{item:C_0} holds for $C_0=\dfrac{1}{\alpha}=\dfrac{1}{n_a-n}$.
	
"(a)$\implies$(b)"	This is the trickiest part of the proof and at the same time is the most important result to use later in the description of equality in \eqref{WSI}. 

Since by assumption, condition {\rm \ref{item:C_0}} holds with $C_0= \frac{1}{n_a-n}$, it follows from  \eqref{eq:explicit-c-0} that 
\begin{equation}\label{c-0-reduced}
\left( \left(\frac{\sigma(y)}{\sigma(x)}\right)^{\tau +n} \left( \frac{\omega(x)}{\omega(y)}\right)^{\alpha +n -p}\right )^{\frac{1}{n(\alpha-\tau) + p \tau}} \leq 
\frac{1}{n_a-n}\,\(\dfrac{1}{p'}\dfrac{\nabla \omega(x)}{\omega(x)}+\dfrac{1}{p}\dfrac{\nabla \sigma(x)}{\sigma(x)}\)\cdot y\ \ \ \textrm{for a.e. $x\in E$ and all}\ y\in E.
\end{equation}
Choosing $y=x$ in \eqref{c-0-reduced} yields
\begin{equation}\label{alpha-tau-1}
1\leq 
\frac{1}{n_a-n}\(\dfrac{\tau}{p'}+\dfrac{\alpha}{p}\).
\end{equation}
Let us recall from the proof of Part (i) that 
$$ n_a - n = \frac{p\, \tau + n(\alpha - \tau)}{\tau - \alpha +p}.$$
This inserted into \eqref{alpha-tau-1} yields
$$ \frac{p\, \tau + n(\alpha - \tau)}{\tau - \alpha +p} \leq \tau + \frac{\alpha - \tau}{p},$$
which is equivalent to
$$ \left( \frac{n+\tau}{\alpha-\tau +p} - \frac{1}{p}\right) (\alpha - \tau) \leq 0.$$
Once again from the expression of $n_a$, the last inequality is 
equivalent to $(\alpha-\tau)(n_a-1)/p\leq 0$. Since $n_a>n\geq 2$, this implies that $\alpha\leq\tau$, and since by assumption $\tau\leq \alpha$, we conclude that  $\alpha= \tau$. 
In particular, we have that $n_a=n+\alpha$ and \eqref{c-0-reduced} reduces to 
\begin{equation}\label{c-0-reduced-bis}
\left( \left(\frac{\sigma(y)}{\sigma(x)}\right)^{\alpha +n} \left( \frac{\omega(x)}{\omega(y)}\right)^{\alpha +n -p}\right )^{\frac{1}{ p \alpha}} \leq 
\frac{1}{\alpha}\,\(\dfrac{1}{p'}\dfrac{\nabla \omega(x)}{\omega(x)}+\dfrac{1}{p}\dfrac{\nabla \sigma(x)}{\sigma(x)}\)\cdot y\ \ \ \textrm{for  a.e. $x\in E$ and all $y\in E$}.
\end{equation}

Let us define the function $f:E\to (0,\infty)$ given by $f(x)=\dfrac{\omega(x)}{\sigma(x)}$, $x\in E.$ Our task is to prove that $f$ is constant on $E$. 
To do this, we first rewrite \eqref{c-0-reduced-bis} in terms of $f$ and $\sigma$ to eliminate $\omega$. In this way we obtain
\begin{equation}\label{c-0-reduced-bis-bis}
\left[ \left(\frac{f(x)}{f(y)}\right)^{\alpha +n -p} \, \left( \frac{\sigma(y)}{\sigma(x)}\right)^{p}\right]^{\frac{1}{p\alpha}} 
\leq \frac{1}{\alpha}\left[ \frac{1}{p} \frac{\nabla \sigma (x)} {\sigma(x)} + \frac{1}{p'} \frac{f (x)\nabla \sigma(x) + \nabla f (x)\sigma(x) }{f(x) \cdot \sigma(x)} \right] \cdot y
\end{equation} 
for a.e $x\in E$ and for all $y\in E$.
Motivated by this inequality,  we define for a.e. $x\in E$ the function $g_x:E\to \mathbb R$  given by
$$g_x(y)=\frac{1}{\alpha}\,\(\dfrac{1}{p'}\dfrac{\nabla f(x)}{f(x)}+\dfrac{\nabla \sigma(x)}{\sigma(x)}\)\cdot y-\(\dfrac{\sigma(y)}{\sigma(x)}\)^\frac{1}{\alpha} \,\(\dfrac{f(x)}{f(y)}\)^\frac{n+\alpha}{\alpha q},\ \ y\in E.
$$
Clearly $g_x$ is continuous in $E$,
and since $\alpha=\tau$ and \eqref{eq:dimensional balance}, (\ref{c-0-reduced-bis-bis}) is equivalent to $g_x(y)\geq 0$ for a.e $x\in E$ and all $y\in E.$ Furthermore, since $f$ is homogeneous of degree zero and differentiable a.e., one has that $\nabla f(x)\cdot x=0$, and thus $g_x(x)=0$ for a.e. $x\in E$. In particular, for a.e. $x\in E$, the function $y\mapsto g_x(y)$  has a global minimum on $E$ at $y=x$ and since 
 $y\mapsto g_x(y)$ is differentiable at $y=x$, we obtain $\nabla g_x(y)|_{y=x}=0.$ 
 This means that for a.e. $x\in E$, one has
 $$\frac{1}{\alpha}\,\(\dfrac{1}{p'}\dfrac{\nabla f(x)}{f(x)}+\dfrac{\nabla \sigma(x)}{\sigma(x)}\)-\frac{1}{\alpha}\frac{\nabla\sigma(x)}{\sigma(x)}
 +\frac{n+\alpha}{\alpha q}\dfrac{\nabla f(x)}{f(x)}=0,$$
which is equivalent to
 $$\(\frac{1}{\alpha p'}+\frac{n+\alpha}{\alpha q}\)\dfrac{\nabla f(x)}{f(x)}=0\ \ {\rm for  \ a.e.}\ x\in E.$$
Since $\dfrac{1}{ p'}+\dfrac{n+\alpha}{ q}>0,$ it follows that 
\begin{equation}\label{nablaf=0}
 \nabla f(x)=0 \ \ \ {\rm for\ a.e.}\ x\in E. 
\end{equation}

 We are going to prove that $f$ is locally Lipschitz in $E$; once we do that, by \eqref{nablaf=0} we may conclude that $f$ is constant. To see this, let $h:E\to (0,\infty)$ be the continuous, a.e. differentiable function  given by 
\begin{equation}\label{h-function-def}
h(x)=\left( \dfrac{\sigma(x)^{\alpha +n}}{\omega(x)^{\alpha +n -p}} \right )^{\frac{1}{ p \alpha}},\ \ x\in E.
\end{equation}
From \eqref{nablaf=0} it follows that $\dfrac{\nabla \omega(x)}{\omega(x)}=\dfrac{\nabla \sigma(x)}{\sigma(x)}$ for a.e. $x\in E$.
A simple computation then shows that 
$\nabla h(x)=\dfrac{1}{\alpha}h(x)\,\dfrac{\nabla \omega(x)}{\omega(x)}\ \ \ \textrm{for \ a.e.}\ x\in E.$ Therefore, 
relation \eqref{c-0-reduced-bis} reduces to 
%\begin{equation}\label{g-ben-derivalt}
%\frac{g(y)}{g(x)}\leq 
%\frac{1}{\alpha}\,\dfrac{\nabla \omega(x)}{\omega(x)}\cdot y\ \ \ \textrm{for a.e. $x\in E$ and all $y\in E$}.
%\end{equation}
%A simple computation based on \eqref{nablaf=0}  shows that 
%$$\nabla g(x)=\frac{1}{\alpha}g(x)\,\dfrac{\nabla \omega(x)}{\omega(x)}\ \ \ \textrm{for \ a.e.}\ x\in E,$$
%	thus \eqref{g-ben-derivalt} reduces to 
\begin{equation}\label{g-ben-derivalt-1}
{h(y)}\leq \nabla h(x)
\cdot y\ \ \ \textrm{for a.e. $x\in E$ and all $y\in E$}.
\end{equation}
Since $h$ is homogeneous of degree one, it follows by  \eqref{g-ben-derivalt-1} that
\begin{equation}\label{g-ben-derivalt-2}
{h(y)}-h(x)\leq \nabla h(x)
\cdot (y-x)\ \ \ \textrm{for a.e. $x\in E$ and all $y\in E$}.
\end{equation}
Now,  Lemma \ref{concave-characterization} implies that $h$ is concave in $E$, thus locally Lipschitz in $E$. By assumption, one of the weights is locally Lipschitz, thus the other one is locally Lipschitz too. 
In particular, $f$ is also locally Lipschitz in $E$, and so from \eqref{nablaf=0} we conclude the proof that $f$ is constant. 
Hence $\omega=c\sigma$ in $E$ for some $c>0$, and so $h(x)=c^\frac{p-\alpha-n}{p\alpha}\sigma(x)^\frac{1}{\alpha}$ for every $x\in E$. Therefore $\sigma^\frac{1}{\alpha}$ is concave in $E$ concluding the proof. 
% 
%{\blue To conclude that $f$ is constant we have two alternatives: assume that the a.e. $x\in E$ statements above are replaced by for all $x\in E$ or assume that the weights $\omega$ and $\sigma$ are locally Lipschitz continuous.
%In the first case, $\nabla f(x)=0$ everywhere in $E$ and so $f$=constant, and in the second case, the function $f$ is locally Lipschitz continuous and so from $\nabla f(x)=0$ a.e. in $x$ we conclude also that $f$ is constant.} 
%
%
%
%The concavity of $\sigma^{1/\alpha}$ in $E$ follows from \eqref{c-0-reduced-bis}. 
%Indeed,   and $\alpha = \tau$, we have that $n_a= n+\alpha$ and so  \eqref{c-0-reduced-bis} reads 
%$$ \left ( \frac{\sigma(y)}{\sigma(x)}\right)^{\frac{1}{\alpha}}\leq \frac{1}{\alpha} \frac{\nabla \sigma (x) }{\sigma (x)} \cdot y \ \ \text{for a.e. $x\in E$ and all $y \in E$}.$$
%By the 1-homogeneity of $\sigma^{\frac{1}{\alpha}}$, it follows that $ \sigma(x) ^{\frac{1}{\alpha}} = \nabla ( \sigma ^{\frac{1}{\alpha}}) (x) \cdot x ,$ for a.e. $x\in E$ 
%and we obtain 
% \begin{equation}\label{eq:concavity of h }
% \sigma(y) ^{\frac{1}{\alpha}} - \sigma(x) ^{\frac{1}{\alpha}}\leq \nabla ( \sigma ^{\frac{1}{\alpha}}) (x) \cdot (y-x)  \ \ \text{for a.e. $x\in E$ and all $y \in E$}.
% \end{equation}
%It remains to apply Lemma  \ref{concave-characterization} to conclude that $h=\sigma^{1/\alpha}$ is concave.  
\end{proof}

{\it Proof of Theorem \ref{main-theorem-1-equality}.}
Let us assume that equality holds in \eqref{WSI} for some $u\in \dot W_\sigma^{1,p}(\mathbb R^n)\setminus \{0\}$, and without loss of generality, assume that $u$ is nonnegative with $$\int_E u(x)^q\,\omega(x)\,dx=1.$$ 
A similar argument as in \cite[Proposition 6]{CENV} implies that $\Delta_{\mathcal D'}\phi$ is absolutely continuous on $E_0$, where $E_0$ denotes the interior of the  set $\{x\in \mathbb R^n:\phi(x)<+\infty\}.$ We notice that $U\cap E={\rm supp}(u)\cap E\subset \overline{E_0}.$

To prove Theorem \ref{main-theorem-1-equality} we discuss separately the equality cases for $p>1$ (see Theorem \ref{main-theorem-1}) and $p=1$ (see Theorem \ref{main-theorem-2}), respectively.

\subsection{Case $p>1$} We split the proof into several cases. 

\textit{Case 1}: condition \ref{item:C_0} holds, $p>1$ and $n_a<+\infty$. \\
Since $u$ gives equality in \eqref{WSI}, we must have equality in each step in the proof of Lemma \ref{lm:condition for two weights}/(i), Case 1. 
In particular, we have equality in the AM-GM inequality  $\det D^2_A\phi(x)\leq
\(\dfrac{\Delta_A \phi(x)}{n}\)^{n}$ for $\mu$-a.e. $x\in E$   (recall that $\mu=u^q\,\omega\,dx$), thus identifying $D^2_A\phi$ with $D^2_{\mathcal D'}\phi$, it turns out that $D^2_A\phi(x)=\lambda I_n$ for a.e. $x\in E$, where $\lambda>0$ and $I_n$ is the $n\times n$-identity matrix. 
Therefore,  for some $x_0\in \mathbb R^n$, one has
\begin{equation}\label{equality-1}
\nabla \phi(x)= \lambda x+x_0\ \ {\rm for \ a.e.}\ \ x\in E\cap E_0.
\end{equation}
Since $\nabla \phi(\overline E)\subseteq \overline E$ and $0\in \overline E$, we necessarily have that $x_0\in \overline E$.

The equality in the second AM-GM inequality in the proof of Lemma \ref{lm:condition for two weights}/(i) yields 
$$
\(\dfrac{\omega\(\nabla \phi(x)\)^{-1/q}\,
	\sigma\(\nabla \phi(x)\)^{1/p}}{\omega(x)^{-n/qn_a}\,\sigma(x)^{n/pn_a}}\)^{n_a/(n_a-n)}
=\dfrac{\Delta_A \phi(x)}{n}\,\omega(x)^{-1/q}\,\sigma(x)^{1/p}\ \ {\rm for \ a.e.}\ \ x\in E\cap E_0.$$
By rearranging the last equation, combined with  $\Delta_A \phi(x)=\lambda n$ for a.e. $x\in E\cap E_0$ and \eqref{equality-1},  it follows that
\begin{equation}\label{equality-2}
{\omega\(\lambda x+x_0\)^{-1/q}\,
	\sigma\(\lambda x+x_0\)^{1/p}}
=\lambda^{(n_a-n)/n_a}\omega(x)^{-1/q}\,\sigma(x)^{1/p}\ \ {\rm for \ a.e.}\ \ x\in E\cap E_0.\end{equation} 

When we apply condition \ref{item:C_0} in the proof of Lemma \ref{lm:condition for two weights}/(i), the equality means that for a.e. $ x\in E\cap E_0$, we have
$$\omega\(\nabla \phi(x)\)^{-1/q}\,\sigma\(\nabla \phi(x)\)^{1/p}= \(C_0\(\dfrac{1}{p'}\dfrac{\nabla \omega(x)}{\omega(x)}+\dfrac{1}{p}\dfrac{\nabla \sigma(x)}{\sigma(x)}\)\cdot \nabla \phi(x)\)^{(n_a-n)/n_a}\,\omega(x)^{-1/q}\sigma(x)^{1/p}.$$
Thus, by \eqref{equality-1} and \eqref{equality-2}, it turns out that 
$$\lambda=C_0 \(\dfrac{1}{p'}\dfrac{\nabla \omega(x)}{\omega(x)}+\dfrac{1}{p}\dfrac{\nabla \sigma(x)}{\sigma(x)}\)\cdot \(\lambda x+x_0\)\ \ {\rm for \ a.e.}\ \ x\in E\cap E_0.$$
By \eqref{eq:homogeneity of weights}, the last relation is equivalent  to 
\begin{equation}\label{equality-3}
1=C_0\(\dfrac{\tau}{p'}+\dfrac{\alpha}{p}+I_0(x)\)\ \ {\rm for \ a.e.}\ \ x\in E\cap E_0,
\end{equation}
where $$I_0(x)=\lambda^{-1}\(\dfrac{1}{p'}\dfrac{\nabla \omega(x)}{\omega(x)}+\dfrac{1}{p}\dfrac{\nabla \sigma(x)}{\sigma(x)}\)\cdot x_0.$$
By using  condition \ref{item:C_0} for $y:=y_k$, where $\{y_k\}_k\subset E$  is a sequence converging to $x_0\in \overline E$, we immediately obtain that $I_0(x)\geq 0$ for a.e. $x\in E.$ Therefore, by \eqref{equality-3} we have that 
\begin{equation}\label{equality-4}
1\geq C_0\(\dfrac{\tau}{p'}+\dfrac{\alpha}{p}\).
\end{equation}
Finally, in the last estimate of the proof of Lemma \ref{lm:condition for two weights}/(i), the equality requires $$C_0\,\(1-\frac{n}{n_a}\)=\dfrac{1}{n_a},$$
i.e.,  
\begin{equation}\label{C-0-value}
C_0=\frac{1}{n_a-n}.
\end{equation}
This means that we have  $\frac{\tau}{p'} + \frac{\alpha}{p} \leq n_a-n$ that is precisely the reverse inequality to \eqref{alpha-tau-1}. 
A similar reasoning as before using  \eqref{equality-4} together with  \eqref{C-0-value} imply 
now the reverse conclusion, that is $\tau\leq \alpha.$

We also notice that $\alpha>0$. Indeed,  if we assume that $\alpha\leq 0$,  we would have  $\tau\leq \alpha\leq 0$ and by picking $y=x\in E$ in \ref{item:C_0}, it follows  
$1\leq C_0\(\frac{\tau}{p'}+\frac{\alpha}{p}\)\leq 0;$ a contradiction. 
 
 Summing up, from the above arguments one concludes that condition \ref{item:C_0} holds with $C_0=\frac{1}{n_a-n}$ and $\tau\leq \alpha$ with $\alpha>0.$ But now from  Proposition \ref{prop-1}/(ii) it follows that there exists $c>0$ such that $\omega(x)=c\sigma(x)$ for every $x\in E$ (thus $\alpha=\tau$ and $n_a=n+\alpha$) and $\sigma^{1/\alpha}$ is concave in $E$. 
 
 Now, we are precisely in the setting of \cite[Theorem A.1]{CFR}. In particular, by 
 the equality in the H\"older inequality, it follows that the extremal function satisfies 
 $|\nabla u(x)|^p\sigma(x)=c_0u(x)^q|x+x_0|^{p'}\omega(x)$ for some $c_0>0$ and every $x\in E$.
 Thus, $|\nabla u(x)|^p=c_0cu(x)^q|x+x_0|^{p'}$,
  obtaining that the extremal function in \eqref{WSI} is 
 $u(x):=u_\gamma(x)=(\gamma+|x+x_0|^{p'})^{-\frac{n+\alpha-p}{p}},$  $\gamma>0$.
 % is determined in such a way that $\int_E u_\lambda(x)^q\,\omega(x)\,dx=1.$ 
We notice that  \eqref{equality-3} reduces to $I_0(x)=0$ for a.e $x\in E$, thus $x_0\in \overline E$ verifies $\nabla \omega(x)\cdot x_0=\nabla \sigma(x)\cdot x_0=0$  for a.e. $x\in E.$ In this way, \eqref{WSI} takes the more familiar form (with only one weight)
\begin{equation}\label{Sob-1}
\(\int_E |u(x)|^q\,\sigma(x) \,dx\)^{1/q}\leq  \tilde K_0\,\(\int_E |\nabla u(x)|^p\,\sigma(x)\,dx\)^{1/p}\ \ \textrm{for all}\ u\in \dot W_\sigma^{1,p}(\mathbb R^n),
\end{equation}
 where 
 \begin{equation}\label{Sob-2}
 \tilde K_0=\frac{p(n_a-1)}{n_a(n_a-p)}\dfrac{\(\int_E u_\gamma(y)^q\,|y|^{p'}\sigma(y)\,dy\)^\frac{1}{p'}\(\int_E u_\gamma(y)^q\sigma(y)\,dy\)^\frac{1}{q}}{\int_E u_\gamma(y)^{q(1-\frac{1}{n_a})}
 	\sigma\(y\)\,dy}
 \end{equation}
 is the best constant in \eqref{Sob-1} (not depending on $\gamma>0$). 

\textit{Case 2}: condition \ref{item:C_0} holds, $p>1$ and $n_a=+\infty$. \\
In order to have equality in \eqref{WSI}, we must have equality in the proof of  Lemma \ref{lm:condition for two weights}/(i), Case 3.  In particular, we have  $\Delta_A \phi(x)= 0$ for a.e. $x\in E,$ which leads us to the degenerate case $\nabla \phi(x)=x_0$ for a.e. $x\in E,$ for some $x_0\in E$, which is not compatible with the Monge-Amp\`ere equation \eqref{eq:MA equation two weights-1-0}. Therefore, no equality can be obtained in \eqref{WSI}. 

\textit{Case 3}: condition \ref{item:C_1} holds and $p>1$. \\
Equality in \eqref{WSI} requires equality in each estimate in the proof of  Lemma \ref{lm:condition for two weights}/(ii), Case 1.
First, as before, the equality in the AM-GM inequality  $\det D^2_A\phi(x)\leq
\(\dfrac{\Delta_A \phi(x)}{n}\)^{n}$ for $\mu$-a.e. $x\in E$ yields
\begin{equation}\label{equality-2-0}
\nabla \phi(x)= \lambda x+x_0\ \ {\rm for \ a.e.}\ \ x\in E\cap E_0
\end{equation}
for some $\lambda>0$ and $x_0\in \overline E$.  The equality in the second estimate, where \eqref{kozbeeso} is applied, together with the continuity of the weights $\sigma$ and $\omega$ implies that 
\begin{equation}\label{kozbeeso-2}
\omega(x)^{\frac{1}{q}}= C_1\sigma(x)^{1/p}\ \ {\rm for\ all}\ x\in E,
\end{equation}
where $C_1>0$ is the constant in condition
\ref{item:C_1}.	Furthermore, the equality when we apply condition
\ref{item:C_1}  requires
$$\(\dfrac{1}{p'}\dfrac{\nabla \omega(x)}{\omega(x)}+\dfrac{1}{p}\dfrac{\nabla \sigma(x)}{\sigma(x)}\)\cdot  (\lambda x+x_0)=0\ \ {\rm for \ a.e.}\ \ x\in E\cap E_0.$$
A similar argument as before shows that  the latter relation can be transformed equivalently into  
	\begin{equation}\label{equality-3-0}
	\dfrac{\tau}{p'}+\dfrac{\alpha}{p}+I_0(x)=0\ \ {\rm for \ a.e.}\ \ x\in E\cap E_0,
	\end{equation}
	where $$I_0(x)=\lambda^{-1}\(\dfrac{1}{p'}\dfrac{\nabla \omega(x)}{\omega(x)}+\dfrac{1}{p}\dfrac{\nabla \sigma(x)}{\sigma(x)}\)\cdot x_0.$$
	By  condition \ref{item:C_1}, it is clear that $\dfrac{\tau}{p'}+\dfrac{\alpha}{p}\geq 0$ (taking $y=x$) and $I_0(x)\geq 0$ for a.e. $x\in E$ (taking $y:=y_k$ where $\{y_k\}_k\subset E$ converges to $x_0$). Therefore, by \eqref{equality-3-0} we have that 
$
	\dfrac{\tau}{p'}+\dfrac{\alpha}{p}= 0$  and $ I_0(x)=0$   for  a.e. $x\in E\cap E_0.
$
Since $n_a=n$, it follows that $\dfrac{\tau}{q}=\dfrac{\alpha}{p}$; this relation combined with $\dfrac{\tau}{p'}+\dfrac{\alpha}{p}= 0$ gives that $\tau=\alpha=0$. 
%Furthermore, relation \eqref{kozbeeso-2} and $I_0(x)=0$   for  a.e. $x\in E\cap E_0$ implies that $\nabla \omega(x)\cdot x_0=\nabla \sigma(x)\cdot x_0=0$  for every $x\in E.$ 

Due to \eqref{kozbeeso-2}, condition \ref{item:C_1} implies that 
\begin{equation}\label{inequality-1-2}
\nabla \omega(x)\cdot y\geq 0,\ \ \nabla \sigma(x)\cdot y\geq 0\ \ {\rm for\ a.e}\ x\in E\ {\rm and\ all}\ y\in E.
\end{equation}
Let $x\in E$ be any differentiability point of $\omega$ and fix $\rho>0$ small enough such that $x+B_\rho\subset E.$ Applying \eqref{inequality-1-2} for $y:=x+z$ with arbitrarily $z\in B_\delta$ and using the fact that $\nabla \omega(x)\cdot x=0$  (since $\tau=0$), it follows that $\nabla \omega(x)\cdot z\geq 0$ for every $z\in B_\delta$. This holds in fact for every $z \in \R^n$,  which implies that $\nabla \omega(x)=0$. 
Since $\omega$ is locally Lipschitz (thanks to our assumption an  \eqref{kozbeeso-2}), the latter relation implies  $\omega$ is a constant, $\omega\equiv c_\omega>0$; in a similar way, $\sigma\equiv c_\sigma>0$. By \eqref{kozbeeso-2}, one has $c_\omega^\frac{1}{q}=C_1c_\sigma^\frac{1}{p}$. We also notice that $x_0$ can be arbitrarily fixed in $\overline E.$ 

A similar argument as in Case 1 shows that when we use  
H\"older inequality in the proof of Theorem \ref{main-theorem-1}/(ii), the equality case implies  that the extremal function verifies 
$|\nabla u(x)|^p=c_1u(x)^q|x+x_0|^{p'}$ for some $c_1>0$ and every $x\in E$. The rest is the same as in \eqref{Sob-1} and \eqref{Sob-2}, where we may choose without loss of generality $\sigma=1;$ in fact, \eqref{Sob-1} is a Talenti-type sharp Sobolev inequality on convex cones. 

%\begin{remark}\label{remark-regularity}\rm 
%	One could expect that the concavity-type condition \ref{item:C_0} together with the homogeneity property of the weights $\omega$ and $\sigma$ imply some more regularity besides the concavity of  $$h(x)=\left( \dfrac{\sigma(x)^{\alpha +n}}{\omega(x)^{\alpha +n -p}} \right )^{\frac{1}{ p \alpha}},\ \ x\in E,$$ see \eqref{h-function-def}.
%	
%	As expected, even if $f=$ \eqref{nablaf=0}
%	
%	I convinced myself that that the C-0 condition a.e. and homogeneity together will not imply a better regularity such as Lipschitz or Sobolev. 
%	The reason is the following: consider two very smooth weights $\sigma$ and $\omega$ satisfying the C-0 condition. Consider the new weights 
%	$\sigma’ = \sigma \cdot f and \omega’ = \omega \cdot f ,$  where is the 0-homogenes version of the Cantor function (handwritten notes sent yesterday).
%	Now we can assume by adding 1 to f that our Cantor function takes its values in [1,2]. Now it is easy to check that the right side of C-0 will stay exactly the 
%	same (as the $\nabla f$ vanishes a.e.). The left side will also be the same up to a multiplicative constant. So C-0 and the homogeneity is satisfied, however,
%	there is no increased regularity that we hoped for. 
%\end{remark}

\subsection{Case $p=1$}
We now turn our attention to analyze the equality cases in Theorem \ref{main-theorem-2}. Since the proof is similar to the case $p>1$, we outline only the differences. 

\textit{Case 1}: condition \ref{item:C_0} holds, $p=1$ and $n_a<+\infty$. \\
We follow the proof of Lemma \ref{lm:condition for two weights}/(i), Case 2. 
First,  for some $\lambda>0$ and $x_0\in \overline E$ one has that
$\nabla \phi(x)= \lambda x+x_0$   for  a.e. $x\in E\cap E_0.$ Similarly to \eqref{equality-2}, one necessarily has that 
$${\omega\(\lambda x+x_0\)^{-1/q}\,
	\sigma\(\lambda x+x_0\)}
=\lambda^{(n_a-n)/n_a}\omega(x)^{-1/q}\,\sigma(x)\ \ {\rm for \ a.e.}\ \ x\in E\cap E_0.$$
Furthermore, it follows that 
$$\lambda=C_0 \dfrac{\nabla \sigma(x)}{\sigma(x)}\cdot \(\lambda x+x_0\)\ \ {\rm for \ a.e.}\ \ x\in E\cap E_0,$$
which can be written as
$$1=C_0(\alpha+I_0(x))\ \ {\rm for \ a.e.}\ \ x\in E\cap E_0,$$
where $I_0(x)=\lambda^{-1}\dfrac{\nabla \sigma(x)}{\sigma(x)}\cdot x_0.$ Since $I_0(x)\geq 0$ for a.e. $x\in E$ (due to condition \ref{item:C_0} for $p=1$), it follows that $C_0\alpha\leq 1.$ Clearly, condition \ref{item:C_0} for $p=1$ and $y=x$ gives that $1\leq C_0\alpha.$ Thus $C_0\alpha= 1.$ On the other hand, we must also have $C_0\(1-\dfrac{n}{n_a}\)=\dfrac{1}{n_a},$
i.e.,  $
C_0=\dfrac{1}{n_a-n}.$ Consequently, we obtain $\dfrac{1}{n_a-n}= \dfrac{1}{\alpha}$, which is equivalent to $(\alpha-\tau)(n+\alpha-p)= 0$. Due to \eqref{eq:p-range}, it follows that $\alpha= \tau$. We can apply again Proposition \ref{prop-1}/(ii) to obtain the existence of $c>0$ such that $\omega(x)=c\sigma(x)$ for every $x\in E$,  and the $\sigma^{1/\alpha}$ is concave in $E$. In this way, \eqref{WSI} reduces to 
\begin{equation}\label{Sob-1-0}
\(\int_E |u(x)|^\frac{n_a}{n_a-1}\,\sigma(x) \,dx\)^{1-\frac{1}{n_a}}\leq  \tilde K_0\,\int_E |\nabla u(x)|\,\sigma(x)\,dx\ \ \textrm{for all}\ u\in \dot W_\sigma^{1,1}(\mathbb R^n),
\end{equation}
where 
\begin{equation}\label{Sob-2-0}
\tilde K_0=\dfrac{1}{n_a}
{\(\int_{B\cap E} \sigma(y)\,dy\)^{-\frac{1}{n_a}}}.
\end{equation}
The constant $\tilde K_0$ in \eqref{Sob-1-0} is sharp. Indeed, according to \cite[rel. (1.14), p. 2977]{Cabre-Ros-Oton-Serra}, one has  
 $P_\sigma(B,E)=(n+\alpha)\int_{B\cap E} \sigma(x)dx$. 
  Since $n_a=n+\alpha,$ by considering  $u(x):=\mathbbm{1}_{B\cap E}(x)$, it yields 
\begin{eqnarray*}
\int_E |\nabla u(x)|\,\sigma(x)\,dx&=&P_\sigma(B,E)=n_a\int_{B\cap E} \sigma(x)dx=\tilde K_0^{-1}\(\int_{B\cap E} \sigma(x) \,dx\)^{1-\frac{1}{n_a}}\\&=&\tilde K_0^{-1}\(\int_E |u(x)|^\frac{n_a}{n_a-1}\sigma(x) \,dx\)^{1-\frac{1}{n_a}},
\end{eqnarray*}
which gives equality in \eqref{Sob-1-0}. 

\textit{Case 2}: condition \ref{item:C_0} holds, $p=1$ and $n_a=+\infty$. \\
 We must have equality in the proof of  Lemma \ref{lm:condition for two weights}/(i), Case 4. Thus we have  $\Delta_A \phi(x)= 0$ for a.e. $x\in E,$ which contradicts again the Monge-Amp\`ere equation \eqref{eq:MA equation two weights-1-0}. Thus, no equality can be obtained in \eqref{WSI}.

\textit{Case 3}: condition \ref{item:C_1} holds and $p=1$. \\
The discussion is similar to Case 3 with $p>1$, obtaining that equality in \eqref{WSI} implies that both $\omega$ and $\sigma$ are constant, $\omega\equiv c_\omega>0$,  $\sigma\equiv c_\sigma>0$, and  $c_\omega^\frac{1}{q}=C_1c_\sigma$, where $C_1>0$ is the constant in condition \ref{item:C_1}. Therefore, \eqref{WSI} becomes the (usual) sharp isoperimetric inequality on the cone $E$. This concludes the proof of Theorem \ref{main-theorem-1-equality}. 
 \hfill $\square$
 
% {\red I have the impression that $P_\sigma(B,E)=(n+\alpha)\int_{B\cap E} \sigma(x)dx$ is a non-trivial result and cannot be derived using simple polar coordinates. We should give more details of the proof here or make a reference, maybe to  ?} 

\section{Examples and applications}\label{sec:examples and applications}

In this section we illustrate the application of Theorems \ref{main-theorem-1}-\ref{main-theorem-1-equality} to various examples. 
% devoted to various examples and applications, where the results described in  can be applied. 

\subsection{Weights satisfying conditions \ref{item:C_0} and \ref{item:C_1}}\label{subsection5-1}
	\subsubsection{Monomial weights}
	We first discuss  the validity of condition \ref{item:C_0} for \textit{monomial} weights to recover from our statements the results of  \cite{Castro}, \cite{Cabre-Ros-Oton}
	%Castro \cite{Castro} 
	and \cite{Nguyen}. 
More precisely, let  $\tau_i\in \mathbb R$ and $\alpha_i\geq 0$, $i=1,...,n$; $\tau=\tau_1+...+\tau_n$ and $\alpha=\alpha_1+...+\alpha_n$ be such that
\begin{equation}\label{assumptions-2-0}
\gamma_i:=\frac{\tau_i}{p'}+\frac{\alpha_i}{p}\geq 0\ \ {\rm and}\ \  \beta_i:=\frac{\alpha_i}{p}- \frac{\tau_i}{q}\geq 0,\ \ i=1,...,n,
\end{equation}
where $q=\frac{p(\tau+n)}{\alpha+n-p}$ with  the property that if $\gamma_i= 0$ for some $i\in \{1,...,n\}$ then $\tau_i=\alpha_i=0$. 

We consider the convex cone  
	\begin{equation}\label{E-set-0}
	E=\left\{x=(x_1,...,x_n)\in \mathbb R^n: x_i>0 \ {\rm whenever}\  \frac{\tau_i}{p'}+\frac{\alpha_i}{p}> 0\right\},
	\end{equation} and the weights $\omega(x)=x_1^{\tau_1}\cdots x_n^{\tau_n}\ \ {\rm  and}\ \ \sigma(x)=x_1^{\alpha_1}\cdots x_n^{\alpha_n}, \ x=(x_1,...,x_n)\in E.$ 

%With this choice of $E$, we have in fact  that $\omega(x)=x_1^{\tau_1}\cdots x_n^{\tau_n}$ and $ \sigma(x)=x_1^{\alpha_1}\cdots x_n^{\alpha_n}$ for every $x=(x_1,...,x_n)\in E$.
%Set 
%$$\gamma_i:=	\frac{\tau_i}{p'}+\frac{\alpha_i}{p}\ \ {\rm and}\ \  \beta_i:=\frac{\alpha_i}{p}-\frac{\tau_i}{q},\ \ i=1,...,n.$$ 
%By (\ref{assumptions-2-0}), one has  $\gamma_i,\beta_i\geq 0,\ \ i=1,...,n$. 

\begin{proposition}\label{prop-0} Assume that $n_a>n.$ Let  $E\subseteq \mathbb R^n$ be the convex cone given in \eqref{E-set-0} and  $\omega(x)=x_1^{\tau_1}\cdots x_n^{\tau_n}$ and $ \sigma(x)=x_1^{\alpha_1}\cdots x_n^{\alpha_n}$ for every $x=(x_1,...,x_n)\in E$. Then condition 
{\rm \ref{item:C_0}} holds with the constant 
\begin{equation}\label{c-0-monomial}
C_0= \frac{n_a}{n_a-n}\(\(\frac{\beta_1}{\gamma_1}\)^{\beta_1}\cdots \(\frac{\beta_n}{\gamma_n}\)^{\beta_n}\)^\frac{n_a}{n_a-n}.
\end{equation}
Here we use the convention $0^0=1.$
\end{proposition} 

\begin{proof} We first assume that  $n_a<\infty.$ Let $x=(x_1,...,x_n)\in E$ and $y=(y_1,...,y_n)\in E$  be fixed. By the scaling invariance relation (\ref{eq:dimensional balance}) and the form of $\beta_i$, we have that 
	$$\beta_1+...+\beta_n+\frac{n}{n_a}=1.$$
	Then, by using the weighted AM-GM inequality, it follows that 
\begin{eqnarray*}
\(\(\dfrac{\sigma(y)}{\sigma(x)}\)^{1/p} \,\(\dfrac{\omega(x)}{\omega(y)}\)^{1/q}\)^{n_a/(n_a-n)}&=& \(\(\frac{y_1}{x_1}\)^{\beta_1}\cdots \(\frac{y_n}{x_n}\)^{\beta_n}\)^{\frac{1}{\beta_1+...+\beta_n}}\\&=&
\(\(\frac{\beta_1}{\gamma_1}\)^{\beta_1}\cdots \(\frac{\beta_n}{\gamma_n}\)^{\beta_n}\)^\frac{1}{\beta_1+...+\beta_n}
\(\frac{\gamma_1}{\beta_1}\frac{y_1}{x_1}\)^{\frac{\beta_1}{\beta_1+...+\beta_n}}\cdots \(\frac{\gamma_n}{\beta_n}\frac{y_n}{x_n}\)^{\frac{\beta_n}{\beta_1+...+\beta_n}}
\\&
\leq&\frac{1}{\beta_1+...+\beta_n}\(\(\frac{\beta_1}{\gamma_1}\)^{\beta_1}\cdots \(\frac{\beta_n}{\gamma_n}\)^{\beta_n}\)^\frac{1}{\beta_1+...+\beta_n}\(\gamma_1\frac{y_1}{x_1}+...+\gamma_n\frac{y_n}{x_n}\)\\&=&
C_0\,\(\dfrac{1}{p'}\dfrac{\nabla \omega(x)}{\omega(x)}+\dfrac{1}{p}\dfrac{\nabla \sigma(x)}{\sigma(x)}\)\cdot y,
\end{eqnarray*}	
	where $$C_0=\frac{1}{\beta_1+...+\beta_n}\(\(\frac{\beta_1}{\gamma_1}\)^{\beta_1}\cdots \(\frac{\beta_n}{\gamma_n}\)^{\beta_n}\)^\frac{1}{\beta_1+...+\beta_n},$$
	which ends the proof. 
	
	When $n_a=+\infty$ (i.e., $p=q$, which is equivalent to $\alpha=p+\tau$), we have that 
$\beta_1 + \ldots + \beta_n =1$. The same proof as before using the AM-GM inequality shows that condition 
	{\rm \ref{item:C_0}} holds (see (\ref{c-0-reduced-0}) in Remark \ref{remark-1}/(ii)) with the constant 
	$$C_0=\(\frac{\beta_1}{\gamma_1}\)^{\beta_1}\cdots \(\frac{\beta_n}{\gamma_n}\)^{\beta_n},$$ which agrees  with \eqref{c-0-monomial} whenever $n_a\to\infty.$
\end{proof}

From the last proposition we have the following corollary of our main theorems. 

\begin{corollary} \label{cor:monomial case}
Let  $\tau_i\in \mathbb R$ and $\alpha_i\geq 0$, $i=1,...,n,$ and $\tau=\tau_1+...+\tau_n$ and $\alpha=\alpha_1+...+\alpha_n$. Consider the convex cone  given  in \eqref{E-set-0}
%	\begin{equation*}
%	E=\left\{x=(x_1,...,x_n)\in \mathbb R^n: x_i>0 \ {\rm whenever}\  \frac{\tau_i}{p'}+\frac{\alpha_i}{p}> 0\right\},
%	\end{equation*}
	 and the weights $\omega(x)=x_1^{\tau_1}\cdots x_n^{\tau_n}\ \ {\rm  and}\ \ \sigma(x)=x_1^{\alpha_1}\cdots x_n^{\alpha_n}, \ x=(x_1,...,x_n)\in E.$ 
	 If conditions \eqref{eq:p-range}, \eqref{eq:p-range-bis} and  \eqref{assumptions-2-0} hold and $q=\frac{p(\tau+n)}{\alpha+n-p}$, then \eqref{WSI} holds.  
In addition, if $\omega = \sigma$ then the constant $K_0$ arising in \eqref{WSI} is optimal.

\end{corollary} 

\begin{proof} 
The first conclusion follows directly from Theorems \ref{main-theorem-1} and  \ref{main-theorem-2} taking into account Proposition \ref{prop-0}. 

To obtain the second conclusion we use Theorem \ref{main-theorem-1-equality}/(i).  Notice that 
	when $\tau_i=\alpha_i$, $i=1,...,n$, one has that $n_a=n+\alpha_1+...+\alpha_n$, $\beta_i=\frac{\alpha_i}{n_a}$ and $\gamma_i=\alpha_i$, $i=1,...,n$, while the convex cone introduced in (\ref{E-set-0}) becomes 
	$$E=\left\{x=(x_1,...,x_n)\in \mathbb R^n: x_i>0 \ {\rm whenever}\  \alpha_i> 0\right\}.$$
	In this case the constant in Proposition \ref{prop-0} reduces to $C_0=\frac{1}{n_a-n}.$ 
\end{proof}

\begin{remark}\rm
The first conclusion of Corollary \ref{cor:monomial case} covers the main result in  \cite[Theorem 1]{Castro} with a slightly different notation.  The second conclusion shows that the main results \cite[Theorem 1.3]{Cabre-Ros-Oton} and \cite[Theorem 4.2, $\theta=1$]{Nguyen} are also particular cases of our results. 
\end{remark}

\begin{remark}\rm Let $E=(0,\infty)^n$ for any $n\geq 2$. 
	
	(a)  If $\omega(x_1,...,x_n)=\sqrt[n]{x_1\cdots x_n}$ and $\sigma(x_1,...,x_n)=x_1+...+x_n$,   $(x_1,...,x_n)\in E,$  the pair $(\omega,\sigma)$ does not satisfy condition \ref{item:C_0}. However, since $\omega\leq \sigma/n$,  Proposition \ref{prop-0} provides (a non-optimal)  \eqref{WSI} for the weights $\omega$ and $\sigma$; the corresponding constant $K_0>0$ in \eqref{WSI} can be obtained by using the monomial setting, see \cite{Cabre-Ros-Oton, Nguyen}. 
	
	(b) Conversely, if $\omega(x_1,...,x_n)=x_1+...+x_n$ and $\sigma(x_1,...,x_n)=\sqrt[n]{x_1\cdots x_n}$,   $(x_1,...,x_n)\in E,$
	it turns out that the pair $(\omega,\sigma)$ satisfies condition \ref{item:C_0} if and only if $n=2$.
%	; in this case we can choose   $K_0:=24^\frac{1}{2}\pi^\frac{1}{6} \Gamma(3/4)^{-2}\approx 3.9482$ in the corresponding \eqref{WSI} with the weights $\omega$ and $\sigma$. 
	%  obtained by using $v(y)=(\lambda+|y|^2)^{-3}$ in the statement of Theorem \ref{main-theorem-1}/(i). 
\end{remark}

\subsubsection{Radial weights}

Using Theorems \ref{main-theorem-1} and \ref{main-theorem-2} as building blocks, we obtain further consequences that are suitable for other applications. The first consequence is the following domain additivity property of \eqref{WSI}.

\begin{corollary} \label{cor:domain additivity}
 
Let $M \in \mathbb{N}$ be a positive integer and assume that $E$ is an open set in $\mathbb{R}^n$ of the form $E = (\cup_{i=1}^M E_i) \cup E_0$, where $E_i$ are pairwise disjoint convex cones for $i= 1, \ldots, M$ and $E_0$ is a set of measure zero. Let $\omega, \sigma : \cup_i E_i \to (0, \infty)$ be two homogeneous weights such that their restrictions $(\omega_{|_{E_i}}, \sigma_{|_{E_i}})$ satisfy the conditions of Theorem \ref{main-theorem-1} or Theorem \ref{main-theorem-2} for all $i = 1, \ldots, M$. Then \eqref{WSI} holds on the set $E$. 

\end{corollary} 

\begin{proof}

Let $u \in C^{\infty}_0(\mathbb{R}^n)$. Applying Theorem \ref{main-theorem-1} or Theorem \ref{main-theorem-2} on the domain $E_i$ we obtain
\begin{equation*}
\(\int_{E_i} |u(x)|^q\,\omega(x) \,dx\)^{1/q}\leq  K_i\,\(\int_{E_i} |\nabla u(x)|^p\,\sigma(x)\,dx\)^{1/p}\ \ \textrm{for all}\ u\in C_0^\infty(\R^n), 
\end{equation*}
for all $i= 1, \ldots M$. 

Since $E_i$ are pairwise disjoint and $E_0$ has measure zero, it follows from Minkowski's inequality that
\begin{align*}
\(\int_{E} |u(x)|^q\,\omega(x) \,dx\)^{1/q}   & =  \(\sum_{i=1}^M\int_{E_i} |u(x)|^q\,\omega(x) \,dx\)^{1/q}  \leq 
\sum_{i=1}^M \(\int_{E_i} |u(x)|^q\,\omega(x) \,dx\)^{1/q} \\  \leq & \sum_{i=1}^M K_i\,\(\int_{E_i} |\nabla u(x)|^p\,\sigma(x)\,dx\)^{1/p}  \leq K_0\,\(\int_{E} |\nabla u(x)|^p\,\sigma(x)\,dx\)^{1/p}, 
\end{align*}
where $K_0 = M^\frac{1}{p'}\max_{i=1,...,M} K_i>0$. 
\end{proof}

With the domain additivity property of \eqref{WSI}, we now consider radial weights and deduce a particular case of the inequality of Caffarelli, Kohn and Nirenberg \cite[Inequality (1.4)]{CKN}, a case also called Hardy-Littlewood-Sobolev's inequality. 
To do this we first prove the following. 

 \begin{corollary}\label{corollary-CKN-1} Let us assume that the parameters $p, q, \alpha, \tau$ satisfy conditions 
 \eqref{eq:p-range}--\eqref{eq:dimensional balance} and $ \frac{\tau}{p'}+\frac{\alpha}{p} >0$. Then there exists $K_0 >0$ such that for all $u \in C^{\infty}_0(\mathbb{R}^n)$ one has 	 	
	\begin{equation} \label{wsi:radial-0}
	\(\int_{\mathbb{R}^n} |u(x)|^q\, |x|^{\tau}  \,dx\)^{1/q}  \leq K_0 \(\int_{\mathbb{R}^n} |\nabla u(x)|^p\,|x|^{\alpha}\,dx\)^{1/p}.
	\end{equation}
	
\end{corollary} 

\begin{proof} 
	By standard arguments we can find $M \in \mathbb{N}$ and pairwise disjoint convex cones $E_i$ 
	$i= 1, \ldots , M$  such that $\mathbb{R}^n = (\cup_{i\in M} E_i)\cup E_0$ where $E_0$ is the union of the boundaries of $E_i$ (and therefore a null measure set). 
	Moreover, we can choose $E_i$ so small that for all $x, y\in E_i$ we have that 
	$x\cdot y  \geq \frac{1}{2} |x| \cdot |y|$. 
	
	Let us assume first that $n_a >n$. Using that $\nabla (|x|^{\alpha}) = \alpha\, x \, |x|^{\alpha-2}$ and 
	$\nabla (|x|^{\tau}) = \tau\, x \, |x|^{\tau-2},$  condition \ref{item:C_0} on $E_i$, $i= 1, \ldots , M,$ can be written as
	\begin{equation} \label{C0-radial}
	\frac{|y|}{|x|} \leq C_0 \, \left(\frac{\tau}{p'} + \frac{\alpha}{p}\right) \, \frac{x\cdot y}{|x|^2}, \ x,y\in E_i.
	\end{equation}
	Using the estimate $x\cdot y  \geq \frac{1}{2} |x| \cdot |y|$, $x,y\in E_i,$ we see that the above relation is satisfied for 
	$x,y \in E_i$ with a properly chosen constant $C_0>0$. The conclusion now follows by Corollary \ref{cor:domain additivity}.
	In the case $n_a =n$ we can argue in a similar way.
\end{proof} 

We notice that the condition $ \frac{\tau}{p'}+\frac{\alpha}{p} >0$ in Corollary \ref{corollary-CKN-1}  is not assumed in \cite{CKN}. However, it turns out that by applying Corollary \ref{corollary-CKN-1} with appropriate values of $\tau,\alpha$ and $q$, we will obtain 
\cite[Inequality (1.4) with $a=1$]{CKN}
for the full range of exponents.
% First, we deduce a particular form of it as a direct application of Corollary \ref{cor:domain additivity}: 
In fact,  with the notation from \cite{CKN}, let $p\geq 1$, $r>0$, $\beta,\gamma\in \mathbb R$ be such that
\begin{equation}\label{CKN-00-1}
\frac{1}{r}+\frac{\gamma}{n}>0,
\end{equation}
\begin{equation}\label{CKN-00-2}
0\leq \beta-\gamma\leq 1
\end{equation}
and 
\begin{equation}\label{CKN-00-3}
\frac{1}{r}+\frac{\gamma}{n}=\frac{1}{p}+\frac{\beta-1}{n}.
\end{equation}
We shall then prove the following desired inequality.

\begin{corollary}\label{corollary-CKN-full} 
	Under assumptions \eqref{CKN-00-1}-\eqref{CKN-00-3}, there exists $K_0=K_0(p,\beta,\gamma)>0$ such that for all $u \in C^{\infty}_0(\mathbb{R}^n)$, one has
	\begin{equation} \label{CKN-full}
	\(\int_{\mathbb{R}^n} |u(x)|^{ r}\, |x|^{\gamma r}  \,dx\)^{1/r}  \leq K_0 \(\int_{\mathbb{R}^n} |\nabla u(x)|^p\,|x|^{\beta p}\,dx\)^{1/p}.
	\end{equation}
	
\end{corollary}

\begin{proof} 
%	We intend to apply  Corollary \ref{corollary-CKN-1} with suitably chosen parameters combined with an appropriate change of variables. To do this,  
Let $d>1$ be fixed that will be specified later, and let 
	$$\tau:=n(d-1)+\gamma r d,\ \ \alpha:=(n-p)(d-1) +\beta p d\ \ \ {\rm and}\ \ q:=r.$$
	We  claim the parameters $p, q, \alpha, \tau$ satisfy conditions \eqref{eq:p-range}, \eqref{eq:p-range-bis} and \eqref{eq:dimensional balance}.
	 First, a straightforward computation shows that the balance condition \eqref{eq:dimensional balance} is equivalent to condition \eqref{CKN-00-3} which determines the value of $r$ in terms of $p$, $\beta$ and $\gamma$. 
%	 . In this way, \eqref{CKN-00-3} becomes the needed version of the balance condition. From here, the value of $r$ is determined by the choices of $p$, $\beta$ and $\gamma$. 
	  
	  Inequality $p<\alpha+n$  in \eqref{eq:p-range} is equivalent to $n+p(\beta-1)>0$ which holds true due to \eqref{CKN-00-1} and \eqref{CKN-00-3}. The second inequality in \eqref{eq:p-range}, i.e., $\alpha\leq\tau+p$, is equivalent to $(\beta-1)p\leq \gamma r$. Adding $n$ to both sides to the last inequality, it follows from \eqref{CKN-00-3} that the resulting inequality is equivalent to $p\leq r$. 
	  Again by  \eqref{CKN-00-3}, $r=\frac{pn}{(\beta-1)p+n-\gamma p}$. Hence $p\leq r$ is equivalent to $\beta-1\leq \gamma$ which holds from  \eqref{CKN-00-2}. Thus \eqref{eq:p-range} holds.  
	  
	  To show \eqref{eq:p-range-bis}, we observe that from  \eqref{CKN-00-3},  
	  $\alpha \geq \(1-\frac{p}{n}\)\tau$ is equivalent to
	  $\frac{\beta}{r}\geq \frac{\gamma}{p}\(1-\frac{p}{n}\)$, which 
	  again by \eqref{CKN-00-3} is equivalent to 
	   $(\beta-\gamma)(\frac{1}{r}+\frac{\gamma}{n})\geq 0$, which in turn holds true from  
	   \eqref{CKN-00-1} and \eqref{CKN-00-2}. 
	
	To apply Corollary \ref{corollary-CKN-1}, it remains to check the inequality $ \frac{\tau}{p'}+\frac{\alpha}{p} >0$, which for the chosen exponents can be written equivalently as $d\(n-1+\beta+\frac{\gamma r}{p'}\)-n+1>0.$ 
From  \eqref{CKN-00-1} and \eqref{CKN-00-3}, it follows that $n-1+\beta+\frac{\gamma r}{p'}=n\(\frac{1}{r}+\frac{\gamma}{n}\)\(1+\frac{r}{p'}\)>0$. So choosing $d>1$ large enough we obtain that $d\(n-1+\beta+\frac{\gamma r}{p'}\)-n+1>0$ as desired. 
	
	Therefore, from Corollary \ref{corollary-CKN-1} there exists $K_0>0$ such that  for every $v \in C^{\infty}_0(\mathbb{R}^n)$ one has
	\begin{equation} \label{wsi:radial-1}
	\(\int_{\mathbb{R}^n} |v(x)|^r\, |x|^{n(d-1)+\gamma r d}  \,dx\)^{1/r}  \leq K_0 \(\int_{\mathbb{R}^n} |\nabla v(x)|^p\,|x|^{(n-p)(d-1) +\beta p d}\,dx\)^{1/p}.
	\end{equation}
	In addition, by an approximation argument, the last inequality is also valid for every $v \in C^1_0(\mathbb{R}^n)$.

	On the other hand, 
for any fixed $d>1$,  if $T:\mathbb R^n\to \mathbb R^n$ is the map defined by $T(x)=|x|^{d-1}x$, then the determinant of its Jacobian is 
	$${\rm det}J_T(x)=d|x|^{n(d-1)},\ x\neq 0,$$
	see Lam and Lu \cite{LL}. For any $u\in C_0^\infty(\mathbb R^n)$ we introduce $Ru(x)=d^{-\frac{1}{p'}}u(T(x))$ (with the usual convention that $\frac{1}{p'}=0$ when $p=1$). 
	Thanks to \cite[Lemma 2.2]{LL}, a change of variable gives that for every $t,\mu\in \mathbb R$ and every continuous function $f:\mathbb R\to \mathbb R$ one has
	\begin{equation}\label{LL-1}
	\int_{\mathbb R^n} \dfrac{f\(d^{-\frac{1}{p'}}u(x)\)}{|x|^t}dx=d\int_{\mathbb R^n} \frac{f\(Ru(x)\)}{|x|^{n+td-nd}}dx
	\end{equation}  
	and
	\begin{equation}\label{LL-2}
	\int_{\mathbb R^n} \frac{|\nabla Ru(x)|^p}{|x|^{d(p+\mu-n)+n-p}}dx\leq \int_{\mathbb R^n} \frac{|\nabla u(x)|^p}{|x|^\mu}dx.
	\end{equation}  
	If we apply \eqref{LL-1} and \eqref{LL-2} with $t:=-\gamma r$, $\mu:=-\beta p$ and $f(s)=|s|^r$, then using  \eqref{wsi:radial-1} with $v:=Ru\in  C^1_0(\mathbb{R}^n)$, we obtain precisely \eqref{CKN-full}.  	
\end{proof} 

\subsubsection{Further examples of weights}
In this section we further illustrate Conditions \ref{item:C_0} and \ref{item:C_1}.
A sufficient condition for \ref{item:C_0} to hold is the following.
%  Another criterion for verifying condition \ref{item:C_0} can be formulated as follows. 
  
  \begin{proposition}\label{prop-c-0}
  	Let $E\subseteq \mathbb R^n$ be an open convex cone and let $\omega,\sigma:E\to (0,\infty)$ be differentiable weights satisfying \eqref{eq:homogeneity of weights}-\eqref{eq:dimensional balance}, and $n_a>n$.  If  $$F(x)=\omega(x)^\delta\,\sigma(x)^\gamma$$ is concave in $E$ with $\delta=-\dfrac{1}{q}\dfrac{n_a}{n_a-n}$, $\gamma=\dfrac{1}{p}\dfrac{n_a}{n_a-n}$ and $\nabla \omega(x)\cdot y\geq 0$ for every $x,y\in E$, then the pair $(\omega,\sigma)$ satisfies condition {\rm \ref{item:C_0}} with constant $C_0=\dfrac{n_a}{n_a-n}$.   
  \end{proposition}

%\rm On account of Proposition \ref{prop-1}/(i) we might ask ourselves what kind of concavity property implies  for two different weights. In order to avoid technicalities, 
		
	\begin{proof}	
	From the form of $F$, the pair $(\omega,\sigma)$ satisfies \ref{item:C_0} if and only if
		$$ \frac{F(y)}{F(x)} \leq C_0 \left (\dfrac{1}{p'} \,\frac{ \nabla \omega(x)}{\omega(x)} +\dfrac{1}{p}\, \frac{\nabla \sigma (x)}{\sigma (x) } \right)\cdot y, \ \ \forall x,y\in E.$$
		To prove the last inequality, we see that by  \eqref{eq:homogeneity of weights}, 
	$F$ is homogenous of degree $\delta\tau+\gamma\alpha$. Hence $\nabla F(x)\cdot x=\(\delta\tau+\gamma\alpha\)\,F(x)$ for every $x\in E.$
	By the concavity of $F$ in $E$, we have that
	\begin{equation}\label{eq:convexity of F}
	F(y)-F(x)\leq \nabla F(x)\cdot (y-x)\ \  \ \textrm{for all}\ x,y\in E.
	\end{equation}
	Since from the balance condition \eqref{eq:dimensional balance} $\delta\tau+\gamma\alpha-1=0$, it follows from \eqref{eq:convexity of F} that
	\begin{equation}\label{eq:consequence of concavity of F}
	\dfrac{F(y)}{F(x)}\leq \(\delta\dfrac{\nabla \omega(x)}{\omega(x)}
	+\gamma\dfrac{\nabla \sigma(x)}{\sigma(x)}\)\cdot y\ \ \ \textrm{for all}\ x,y\in E.
	\end{equation}
	On the other hand, by assumption $\nabla \omega(x)\cdot y\geq 0$ and $\delta<0$, so we get  
	\[
	\dfrac{F(y)}{F(x)}\leq 
	\(\dfrac{1}{p}\dfrac{n_a}{n_a-n}\dfrac{\nabla \sigma(x)}{\sigma(x)}\)\cdot y\ \ \ \textrm{for all}\ x,y\in E.
	\]
	Thus, $\nabla \sigma(x)\cdot y\geq 0$ for every $x,y\in E$. 
	Using again that $\nabla \omega(x)\cdot y\geq 0$ for every $x,y\in E$, we obtain from \eqref{eq:consequence of concavity of F} that \ref{item:C_0} holds with  $C_0=\dfrac{n_a}{n_a-n}.$
	\end{proof}

To illustrate Proposition \ref{prop-c-0} we show the following example. 
Let $E=(0,\infty)^n$ with $n\geq 2$, $0<\alpha<p$ and $1\leq  p<\alpha+n$.
If $\omega\equiv 1$, $\sigma(x)=\(\dfrac{x_1\cdots x_n}{x_1+...+x_n}\)^{\alpha/(n-1)}$, then $n_a=\dfrac{pn}{p-\alpha}$,
$$F(x)=\omega(x)^\delta\,\sigma(x)^\gamma=\(\frac{x_1\cdots x_n}{x_1+...+x_n}\)^\frac{1}{n-1}$$ 
in Proposition \ref{prop-c-0}, which is concave in $E$, see Marcus and Lopes \cite{MM}. 
Therefore, the pair $(\omega,\sigma)$ satisfies \ref{item:C_0} with $C_0=p/\alpha$ and from Theorem \ref{main-theorem-1}, \eqref{WSI} holds for these weights with $q=\dfrac{p\,n}{\alpha+n-p}$.

We conclude this part by giving an example of weights for which condition \ref{item:C_1} holds.
Let $E=(0,\infty)^n$, $n\geq 2$,  $\tau\geq 0$, and $1\leq p<n$.
If $\omega(x)=(x_1+...+x_n)^\tau$ and $\sigma(x)=|x|^{\tau(1-p/n)}$, then $n_a=n$ and $q=\frac{np}{n-p}$. 
Since $$\sup_{x\in E}\dfrac{\omega(x)^{1/q}}{\sigma(x)^{1/p}}=n^\frac{\tau}{2q}\in (0,\infty),$$
condition \ref{item:C_1} holds, and from Theorem \ref{main-theorem-1}/(ii)  we get that \eqref{WSI} holds for these weights. 
In particular, if $\tau=0,$ then \eqref{WSI} reduces to the sharp Sobolev inequality of Talenti \cite{Talenti} on the cone $E.$

\newpage

 \subsection{Weighted PDEs}\label{eigenvalue-1}

\subsubsection{Spectral gap} In this subsection we provide an estimate of the spectral gap for a weighted eigenvalue problem. More precisely, we have the following.
 
\begin{proposition}\label{spectral-gap}
	Let  $E\subseteq \mathbb R^n$ be an open convex cone and let $\Omega\subset \mathbb R^n$ be an open bounded set such that $\Omega\cap E\neq \emptyset$. Let  $\omega,\sigma:\overline E\to [0,\infty)$ be two continuous nonzero weights  which are differentiable in $E$, satisfying \eqref{eq:homogeneity of weights} with $\alpha=\tau+2$, condition {\rm \ref{item:C_0}} for some $C_0>0$, and $\sigma|_{\partial E}=0$. Then any eigenvalue $\lambda>0$ of the problem 
	\[ \   \left\{ \begin{array}{lll}
	-{\rm div}(\sigma \nabla u)=\lambda \omega u& {\rm in} &   \Omega\cap E, \\
	% u\geq 0 &\mbox{in} &   \Omega;\\
	u=0  & \text{on}& \partial\Omega\cap \overline E. 
	\end{array}\right. \eqno{({P})}\]
verifies
	$$\lambda\geq \frac{1}{4C_0^2}\sup_{v\in C_0^\infty(\Omega)\setminus \{0\},v\geq 0}\dfrac{\(\int_E v(y)\omega\(y\)^{-\frac{1}{2}}\,
		\sigma\(y\)^{\frac{1}{2}}\,dy\)^2}{\int_E v(y)\,|y|^2dy \int_E v(y)dy}>0.$$ 
\end{proposition}
 
 	\begin{proof}
 	Let us multiply the first equation in $(P)$ by $u\neq 0$; an integration and the divergence theorem gives that 
\begin{equation}\label{eigenvalue-0}
 -\int_{\partial(\Omega\cap E)}\sigma(x) \dfrac{\partial u}{\partial \textbf{n}}(x)u(x)ds(x)+	\int_{\Omega}|\nabla u(x)|^2\sigma(x)dx =\lambda \int_{\Omega}u(x)^2\omega(x)dx.
	\end{equation}
 	Since $\partial(\Omega\cap E)\subseteq \partial\Omega \cup \partial E$, the first integral in the left hand side vanishes either for $\sigma|_{\partial E}=0$ or for the Dirichlet boundary condition $u=0$ on $\partial\Omega\cap \overline E$. Therefore, equation \eqref{eigenvalue-0} reduces to 
 		\begin{equation}\label{eigenvalue-2}
	\int_{\Omega}|\nabla u(x)|^2\sigma(x)d =\lambda \int_{\Omega}u(x)^2\omega(x)dx.
 	\end{equation}
 	 Since  $\tau+n>0$ (by the locally integrability of $\omega$) and $\alpha=\tau+2$, assumptions (\ref{eq:p-range})-(\ref{eq:dimensional balance}) are immediately verified with the choices $p=q=2.$ In particular, $n_a=+\infty$ and we can apply Theorem \ref{main-theorem-1}/(i), obtaining 
 	 $$\lambda=\dfrac{\displaystyle\int_{\Omega}|\nabla u(x)|^2\sigma(x)dx}{\displaystyle\int_{\Omega}u(x)^2\omega(x)dx}\geq K_0^{-2},$$
 	 where the constant $K_0>0$ appears in the statement of  Theorem \ref{main-theorem-1}/(i). The rest is a simple computation. 
\end{proof}

\begin{remark}\rm 
	Due to \eqref{eigenvalue-0}, a similar spectral gap estimate can be obtained in the same way also for the Neumann boundary value condition. Moreover, the case $p\neq 2$ can be also handled using the operator ${\rm div}(\sigma |\nabla u|^{p-2}\nabla u)$ in problem ({\it P}). 
\end{remark}

%{\red Can we apply our inequality for general $p$ the get spectral gap estimate for nonlinear eigenvalue problem as well?} 
\subsubsection{A variational problem}
%\section{Final comments and open questions}\label{sec:final comments and open questions}

Applying a variational method we prove the following result.

\begin{proposition}\label{alkalmazas}  

Let  $E\subseteq \mathbb R^n$ be an open convex cone and let $\Omega\subset \mathbb R^N$ be an open bounded set such that $\Omega\cap E\neq \emptyset$. 
Let $\omega,\sigma:\overline E\to [0,\infty)$ be two nonzero weights continuous in $\overline E$, differentiable a.e. in $E$, and satisfying \eqref{eq:homogeneity of weights}-\eqref{eq:dimensional balance} with $\alpha<\tau+2$,  condition {\rm \ref{item:C_0}}, and $\sigma|_{\partial E}=0$. Then for every $r\in (2,q)$ the problem 
	\[ \   \left\{ \begin{array}{lll}
	 -{\rm div}(\sigma \nabla u)+\sigma u= \omega u^{r-1}& {\rm in} &   \Omega\cap E, \\
	% u\geq 0 &\mbox{in} &   \Omega;\\
	u\geq 0& {\rm in} &   \Omega\cap E,\\
	u= 0& {\rm on} &   \partial\Omega\cap \overline E,\\
	\end{array}\right. \eqno{({\mathcal P})}\]
has a nonzero weak solution in the weighted Sobolev space $W_\sigma^{1,2}(\Omega)$. 	
\end{proposition}

\begin{proof}
	 We first recall that the weighted Sobolev space $W_\sigma^{1,2}(\Omega)$ is the set of all measurable functions such that $u\in L^2_\sigma(\Omega\cap E)$ and $|\nabla u|\in L^2_\sigma(\Omega\cap E)$ with the norm
	 $$\|u\|_{W_\sigma^{1,2}(\Omega)}=\(\int_{\Omega\cap E}|\nabla u(x)|^2\sigma(x) dx+\int_{\Omega\cap E} u(x)^2\sigma(x)dx\)^{1/2}.$$ By our assumptions,  Theorem \ref{main-theorem-1} implies that the space $W_\sigma^{1,2}(\Omega)$  is continuously embedded into $L^q_\omega(\Omega\cap E)$, where $q=\dfrac{2(\tau+n)}{\alpha+n-2}$ is the critical exponent. 
	 We also notice that 	 
	 $2<q$ since $\alpha<\tau+2$. Thus,  it follows from the boundedness of $\Omega$ that $W_\sigma^{1,2}(\Omega)$  is compactly embedded into $L^r_\omega(\Omega\cap E)$ for any $r\in (2,q)$.

%{\red The compact embedding  for $r<q$ should follow by standard methods using the continuity of the weights but I need to check this.
%Should we say here something more?} 	 
	
Fix $r\in (2,q)$. Instead of   $({\mathcal P})$ we consider first the problem
%
%	Let $h,H:\mathbb R\to [0,\infty)$ be defined by $h(s)=s_+^{r-1}$ and $H(s)=\frac{s_+^{r}}{r},$ where $s_+=\max(s,0)$ and associate with problem 
	\[ \   \left\{ \begin{array}{lll}
-{\rm div}(\sigma \nabla u)+\sigma u= \omega u_+^{r-1}& {\rm in} &   \Omega\cap E, \\
% u\geq 0 &\mbox{in} &   \Omega;\\
u= 0& {\rm on} &   \partial\Omega\cap \overline E,\\
\end{array}\right. \eqno{({\mathcal P_+})}\]
where we used the notation $u_+ = \max \{ u, 0\}$. 

The energy functional $\mathcal E:W_\sigma^{1,2}(\Omega)\to \mathbb R$ associated with  problem $({\mathcal P_+})$ is  defined by
	$$\mathcal E(u)=\frac{1}{2}\|u\|^2_{W_\sigma^{1,2}(\Omega)}-\frac{1}{r}\int_{\Omega\cap E}(u(x))_+^r\omega(x)dx.$$
	Standard arguments imply that $\mathcal E$ is well-defined (since $W_\sigma^{1,2}(\Omega)$  is continuously embedded into $L^r_\omega(\Omega\cap E)$)  and $\mathcal E\in C^1(W_\sigma^{1,2}(\Omega); \mathbb R)$; moreover,  its differential  is given by 
	$$\mathcal E'(u)(v)=\int_{ \Omega\cap E}(\nabla u(x)\cdot \nabla v(x)+u(x)v(x))\sigma(x)dx-\int_{\Omega\cap E}(u(x))_+^{r-1}\omega(x)v(x)dx, $$ 
	 for all $u,v\in W_\sigma^{1,2}(\Omega).$ In fact, using the divergence theorem together with the Dirichlet boundary condition $u=0$ on $\partial\Omega\cap \overline E$ and $\sigma|_{\partial E}=0$, it follows that 
	 $$\mathcal E'(u)(v)=\int_{ \Omega\cap E}(-{\rm div}(\sigma(x) \nabla u(x))+\sigma(x)u(x))v(x)dx-\int_{\Omega\cap E}\omega(x)(u(x))_+^{r-1}v(x)dx.$$ 
	 In particular, $u\in W_\sigma^{1,2}(\Omega)$ is a critical point of $\mathcal E$ if and only if $u$ is a weak solution of problem $({\mathcal P_+})$.

	We are going to prove that $\mathcal E$ satisfies the Palais-Smale condition on $W_\sigma^{1,2}(\Omega)$. In order to complete this, we consider a sequence  $\{u_k\}_k\subset W_\sigma^{1,2}(\Omega)$ such that  $\mathcal E'(u_k)\to 0$ as $k\to \infty$ and $|\mathcal E(u_k)|\leq C$ ($k\in \mathbb N$) for some $C>0,$ and our aim is to prove that there exists a subsequence of $\{u_k\}_k$ which converges strongly in $W_\sigma^{1,2}(\Omega)$ to some element $u\in W_\sigma^{1,2}(\Omega)$. 
	We notice that 
	$$r\mathcal E(u_k)- \mathcal E'(u_k)(u_k)= \left(\frac{r}{2}-1\right)\|u_k\|^2_{W_\sigma^{1,2}(\Omega)} ,\ \ k\in \mathbb N. $$
Since $\mathcal E'(u_k) \to 0$, we have $|\mathcal E'(u_k)(u_k)| \leq 1$ for large enough values of $k$. 
Therefore, for large $k\geq 1$ one has that $|r\mathcal E(u_k)- \mathcal E'(u_k)(u_k)|\leq rC+\|u_k\|_{W_\sigma^{1,2}(\Omega)}$. Because $r>2$, the latter relation implies that 
	 $\{u_k\}_k$ is bounded in $W_\sigma^{1,2}(\Omega)$. In particular,  we may extract a  subsequence of $\{u_k\}_k$ (denoted in the same way) which  converges weakly to an element $u\in W_\sigma^{1,2}(\Omega)$, and strongly to $u$ in $L^r_\omega(\Omega\cap E)$. The latter follows from the fact that $W_\sigma^{1,2}(\Omega)$  is compactly  embedded into $L^r_\omega(\Omega\cap E)$.  A simple computation shows that 
	$$\|u_k-u\|^2_{W_\sigma^{1,2}(\Omega)}=\mathcal E'(u_k)(u_k-u)-\mathcal E'(u)(u_k-u)+\int_{\Omega\cap E}((u_k)_+^{r-1}-u_+^{r-1})(u_k-u)\omega dx,\ \ k\in \mathbb N.$$
    Since $\mathcal E'(u_k)\to 0$ as $k\to \infty$ and $\{u_k\}_k$ is bounded in $W_\sigma^{1,2}(\Omega)$, one has that $\mathcal E'(u_k)(u_k-u)\to 0$ as $k\to \infty$. Since $\{u_k\}_k$ converges weakly to  $u$, one has that $\mathcal E'(u)(u_k-u)\to 0$ as $k\to \infty$. Moreover, since $\{u_k\}_k$ converges  strongly to $u$ in $L^r_\omega(\Omega\cap E)$,  H\"older's inequality implies that 
	$$\left|\displaystyle\int_{\Omega\cap E}((u_k)_+^{r-1}-u_+^{r-1})(u_k-u)\omega dx\right|\leq \(\|u_k\|^{r-1}_{L^r_\omega(\Omega\cap E)}+\|u\|^{r-1}_{L^r_\omega(\Omega\cap E)}\) \|u_k-u\|_{L^r_\omega(\Omega\cap E)} \to 0$$ as $k\to \infty$. Summing up, it follows that $\|u_k-u\|^2_{W_\sigma^{1,2}(\Omega)}\to 0$ as $k\to \infty,$ which means that  $\{u_k\}_k$ strongly converges to $u$ in $W_\sigma^{1,2}(\Omega)$.

	We shall prove that  $\mathcal E$ satisfies the mountain pass geometry, i.e., there exist $w_0\in W_\sigma^{1,2}(\Omega)$ and $\rho>0$ such that $\|w_0\|_{W_\sigma^{1,2}(\Omega)}>\rho$ and 
	\begin{equation}\label{MPG-1}
		\inf_{\|u\|_{W_\sigma^{1,2}(\Omega)}=\rho}\mathcal E(u)>\mathcal E(0)\geq \mathcal E(w_0).
	\end{equation}
	 To see this, let $c_{\omega,\sigma}>0$ be the constant in the Sobolev embedding  
	$W_\sigma^{1,2}(\Omega)$  into $L^r_\omega(\Omega\cap E)$, i.e., $\|u\|_{L^r_\omega(\Omega\cap E)}\leq c_{\omega,\sigma} \|u\|_{W_\sigma^{1,2}(\Omega)}$ for every $u\in W_\sigma^{1,2}(\Omega).$ Therefore, since $\|u_+\|_{L^r_\omega(\Omega\cap E)}\leq \|u\|_{L^r_\omega(\Omega\cap E)}$, it follows that
	\begin{eqnarray}\label{MPG-2}
		\mathcal E(u)&=&\frac{1}{2}\|u\|^2_{W_\sigma^{1,2}(\Omega)}-\frac{1}{r}\int_{\Omega\cap E}(u(x))_+^r\omega(x)dx=\frac{1}{2}\|u\|^2_{W_\sigma^{1,2}(\Omega)}-\frac{1}{r}\|u_+\|_{L^r_\omega(\Omega\cap E)}^r\nonumber \\
		&\geq &\frac{1}{2}\|u\|^2_{W_\sigma^{1,2}(\Omega)}-\frac{1}{r}\|u\|_{L^r_\omega(\Omega\cap E)}^r\nonumber\\
		&\geq & \frac{1}{2}\|u\|^2_{W_\sigma^{1,2}(\Omega)}-\frac{1}{r}c_{\omega,\sigma}^r \|u\|_{W_\sigma^{1,2}(\Omega)}^r=\(\frac{1}{2}-\frac{1}{r}c_{\omega,\sigma}^r \|u\|_{W_\sigma^{1,2}(\Omega)}^{r-2}\) \|u\|^2_{W_\sigma^{1,2}(\Omega)}.
	\end{eqnarray}
Since $r>2$, the number 
$\rho:=\(\frac{r}{4c_{\omega,\sigma}^r}\)^\frac{1}{r-2}$ is well defined and $\rho>0$. 
Thus, for any $u\in W_\sigma^{1,2}(\Omega)$ with $\|u\|_{W_\sigma^{1,2}(\Omega)}=\rho$, the  estimate \eqref{MPG-2} gives that 
$$\mathcal E(u)\geq \(\frac{1}{2}-\frac{1}{r}c_{\omega,\sigma}^r \rho^{r-2}\) \rho^2=\frac{\rho^2}{4}.$$
Therefore, since $\mathcal E(0)=0$, the left hand side of \eqref{MPG-1} immediately holds. 
%	$$\inf_{\|u\|_{W_\sigma^{1,2}(\Omega)}=\rho}\mathcal E(u)>0=\mathcal E(0).$$

On the other hand, let $w\in W_\sigma^{1,2}(\Omega)$ be any nonnegative, nonzero function. %For any $t>0$, one has that 
%	$$\mathcal E(tw)=\frac{t^2}{2}\|w\|^2_{W_\sigma^{1,2}(\Omega)}-\frac{t^r}{r}\|w\|_{L^r_\omega(\Omega\cap E)}^r.$$
	Since $r>2$, we may fix $t_0>0$ large enough such that $$t_0>\max\left\{\dfrac{\rho}{\|w\|_{W_\sigma^{1,2}(\Omega)}},\(\dfrac{r\|w\|^2_{W_\sigma^{1,2}(\Omega)}}{2\|w\|_{L^r_\omega(\Omega\cap E)}^r}\)^\frac{1}{r-2}\right\}.$$
Accordingly, the function  $w_0:=t_0w\in  W_\sigma^{1,2}(\Omega)$ verifies  $\|w_0\|_{W_\sigma^{1,2}(\Omega)}>\rho$ and 
	$$\mathcal E(w_0)=\mathcal E(t_0w)=\frac{t_0^2}{2}\|w\|^2_{W_\sigma^{1,2}(\Omega)}-\frac{t_0^r}{r}\|w\|_{L^r_\omega(\Omega\cap E)}^r<0,$$
	which is the right hand side of \eqref{MPG-1}.

	We are now in a position to apply the Mountain Pass Theorem, see e.g. Rabinowitz \cite{Rabinowitz},  which implies the existence of a critical point $u\in W_\sigma^{1,2}(\Omega)$ of $\mathcal E$ with the property that $\mathcal E(u)>0$ (thus $u\neq 0$), which is a weak solution to the problem $({\mathcal P_+})$.
	
It remains to prove that $u$ is nonnegative and weakly solves the original problem $({\mathcal P})$.  By multiplying the first equation of $({\mathcal P_+})$   by $u_-=\min(u,0)$, an integration on $\Omega\cap E$ implies that 
	$\|u_-\|_{W_\sigma^{1,2}(\Omega)}=0$, i.e. $u_-=0$. Accordingly, $u\geq 0$   is a nonzero weak solution to the original problem $({\mathcal P})$ as well, which completes the proof.
\end{proof}

{\section{Final comments and open questions}\label{sect-open-question} 

\subsection{Necessity of conditions \eqref{eq:p-range}, \eqref{eq:p-range-bis} and \eqref{eq:dimensional balance}}\label{sec:necessity of the conditions for wsi} 
 
We start this section showing that by choosing appropriate test functions in  \eqref{WSI},  conditions \eqref{eq:p-range}, \eqref{eq:p-range-bis}, and \eqref{eq:dimensional balance} on the parameters 
are \textit{necessary} for the validity of \eqref{WSI}.
  
Condition \eqref{eq:dimensional balance} follows by scaling: if $u$ verifies \eqref{WSI}, then $u_\lambda(x)=u(\lambda x)$ also satisfies \eqref{WSI} for each $\lambda>0$. 
Also, since $q>0$, the left hand side inequality in \eqref{eq:p-range} follows immediately from \eqref{eq:dimensional balance} because $\tau+n>0$ from the local integrability of $\omega$.   

Let us next prove the right hand inequality in \eqref{eq:p-range}.
Let $\varphi$ be a smooth function defined for $t\geq 0$ satisfying $\p(t)=0$ for $0\leq t<1$, $\p(t)=1$ for $t\geq 2$, and $0\leq \p(t)\leq 1$ for all $t>0$. Also choose $h(t)$ smooth for $t\in \R$ with $h(t)=1$ for $|t|\leq 1$, $h(t)=0$ for $|t|\geq 2$ and $0\leq h\leq 1$. 
Given $\epsilon>0$, the function $u_\epsilon(x)=|x|^{-\beta}\,\log|x|\,\p(|x|/\epsilon)\,h(|x|)$ belongs to $C_0^\infty(\R^n)$ with support in the ring $\{\epsilon\leq |x|\leq 2\}$ for each $\beta\in \R$ and so $u_\epsilon$ satisfies \eqref{WSI}.
If $\beta=(\tau+n)/q$, we have for $\epsilon<1/2$ that
 \begin{align*}
 \int_E |u_\epsilon(x)|^q\,\omega(x) \,dx
 &\geq
 \int_{E\cap \{2\epsilon\leq |x|\leq 1\}} |x|^{-\beta q}\(\log \frac{1}{|x|}\)^q\,\omega(x) \,dx\\
 &=
 \int_{2\epsilon}^1t^{-\beta q+\tau+n-1}\(\log \frac{1}{t}\)^q\,dt \,\int_{E\cap S^{n-1}}\omega(x)\,dx\\
% &=
%  \int_{2\epsilon}^1t^{-1}\(\log \frac{1}{t}\)^q\,dt \,\int_{E\cap S^{n-1}}\omega(x)\,dx\quad \\
%  &=
%  \left.-\dfrac{1}{q+1}\(\log \dfrac{1}{t}\)^{q+1}\right|_{t=2\epsilon}^{t=1}\,\int_{E\cap S^{n-1}}\omega(x)\,dx\\
  &=
  \dfrac{1}{q+1}\,\(\log \dfrac{1}{2\epsilon}\)^{q+1}\,\int_{E\cap \mathbb S^{n-1}}\omega(x)\,dx.
  \end{align*}
 Let us now estimate $\int_E |\nabla u_\epsilon(x)|^p\,\sigma(x) \,dx$ from above.
 We have
 \begin{align*}
 \nabla u_\epsilon (x)&=(-\beta)|x|^{-\beta-1}\dfrac{x}{|x|}\,\log|x|\,\p(|x|/\epsilon)\,h(|x|)
 +
 |x|^{-\beta}\dfrac{1}{|x|}\dfrac{x}{|x|}\p(|x|/\epsilon)\,h(|x|)\\
 &\qquad 
 +|x|^{-\beta}\log|x|\, \p'(|x|/\epsilon)\dfrac{1}{\epsilon}\dfrac{x}{|x|}h(|x|)
 +|x|^{-\beta}\log|x|\, \p\(|x|/\epsilon\)\, h'(|x|)\dfrac{x}{|x|}.
 \end{align*}
 Hence
 \begin{align*}
 |\nabla u_\epsilon (x)|
 &\leq |\beta|\,|x|^{-\beta-1}\,|\log|x||
 \,\chi_{\epsilon\leq |x|\leq 2}(x)
 +
 |x|^{-\beta-1}\,\chi_{\epsilon\leq |x|\leq 2}(x)\\
 &\qquad 
 +\|\p'\|_\infty\,|x|^{-\beta}|\log|x||\, \dfrac{1}{\epsilon}\,\chi_{\epsilon\leq |x|\leq 2\epsilon}(x)
 +|x|^{-\beta}\,|\log|x||\, \|h'\|_\infty\,\,\chi_{1\leq |x|\leq 2}(x)\\
% &\leq
%|\beta|\,|x|^{-\beta-1}\,|\log|x||
% \,\chi_{\epsilon\leq |x|\leq 2}(x)
% +
% |x|^{-\beta-1}\,\chi_{\epsilon\leq |x|\leq 2}(x)\\
% &\qquad 
% +2\,\|\p'\|_\infty\,|x|^{-\beta-1}|\log|x||\, \,\chi_{\epsilon\leq |x|\leq 2\epsilon}(x)
% +2\,|x|^{-\beta-1}\,|\log|x||\, \|h'\|_\infty\,\,\chi_{1\leq |x|\leq 2}(x)\\ 
 &\leq
 C_1\,|x|^{-\beta-1}\(1+|\log |x||\)\,\chi_{\epsilon\leq |x|\leq 2}(x),
  \end{align*}
with $C_1>0$ a constant depending only on $\beta,\|h'\|_\infty$, and $\|\p'\|_\infty$.
Therefore 
\begin{align*}
\int_E |\nabla u_\epsilon(x)|^p\,\sigma(x)\,dx
&=
\int_{E\cap \{\epsilon\leq |x|\leq 2\}} |\nabla u_\epsilon(x)|^p\,\sigma(x)\,dx\\
&\leq 
C_1^p\,\int_{E\cap \{\epsilon\leq |x|\leq 2\}}
|x|^{-(\beta+1)p}\(1+|\log |x||\)^p\,\sigma(x)\,dx
:=C_1^p\,I.
\end{align*}
Integrating in polar coordinates
\[
I=\int_\epsilon^2 t^{-(\beta+1)p+n-1+\alpha}\(1+|\log t|\)^p\,dt\,\int_{E\cap S^{n-1}}\sigma(x)\,dx,
\]
and from \eqref{eq:dimensional balance} and the choice of $\beta$, the exponent $-(\beta+1)p+n-1+\alpha=-1$. So
\begin{align*}
I&=C\,\int_\epsilon^2 t^{-1}\(1+|\log t|\)^p\,dt
\leq 2^p\,C\,\int_\epsilon^2 t^{-1}\(1+|\log t|^p\)\,dt\\
&=2^p\,C\,\(\int_\epsilon^2 t^{-1}\,dt +\int_\epsilon^2 t^{-1} |\log t|^p\,dt\)=2^p\,C\,\(I_1+I_2\).
\end{align*}
Now $I_1=\log(2/\epsilon)$ and
\begin{align*}
I_2&=\int_\epsilon^1 t^{-1} |\log t|^p\,dt+\int_1^2 t^{-1} |\log t|^p\,dt\\
&=
\int_\epsilon^1 t^{-1} \(\log \dfrac{1}{t}\)^p\,dt+c_p
=
\dfrac{1}{p+1} \(\log \dfrac{1}{\epsilon}\)^{p+1}+c_p.
\end{align*}
We then obtain the estimate
\[
\int_E |\nabla u_\epsilon(x)|^p\,\sigma(x)\,dx
\leq
C_p\,\(\(\log \dfrac{1}{\epsilon}\)^{p+1}+\log \dfrac{2}{\epsilon}+1\)
\]
and since $u_\epsilon$ satisfies \eqref{WSI} it then follows from the estimate of the $L^q$-norm of $u_\epsilon$ that 
\[
\(\log \dfrac{1}{2\epsilon}\)^{1+\frac{1}{q}}
\leq 
C\, \(\(\log \dfrac{1}{\epsilon}\)^{p+1}+\log \dfrac{2}{\epsilon}+1\)^{1/p},
\]
for all $\epsilon$ small with $C$ independent of $\epsilon$.
Since the dominant term, as $\epsilon \to 0$, on the right hand side of the last inequality is $\(\log \dfrac{1}{\epsilon}\)^{1+\frac{1}{p}}$, we then get that $p\leq q$ which together with \eqref{eq:dimensional balance} yields the inequality on right hand side of \eqref{eq:p-range}.

It remains to prove that \eqref{eq:p-range-bis} is necessary for \eqref{WSI}. 
Fix $y_0\in E\cap \mathbb S^{n-1}$. The idea is to construct a test function supported on a small ball whose center is along the direction $y_0$ that tends to infinity. Since $E$ is open, we may pick $r_0>0$ small enough with $\overline B_{r_0}(y_0)\subset E$. 
Let $$m_0:=\min_{\overline B_{r_0}(y_0)}\omega>0,\ \ M_0:=\max_{\overline B_{r_0}(y_0)}\sigma>0,$$ fix a function $v\in C_0^\infty(B_1)\setminus\{0\}$, and define $u_\delta(x)=v(x-\delta y_0)$ for  $\delta>0$.  Note that $u_{\delta}  \in C_0^{\infty}\(B_1(\delta y_0)\)$.  
Observe also, that if $\delta\,r_0>1$, then $B_1\(\delta y_0\)\subset B_{\delta r_0}\(\delta y_0\)\subset \delta(\overline B_{r_0}(y_0))\subset E$, since $E$ is a cone.
Therefore, by (\ref{eq:homogeneity of weights}) and the definitions of $m_0,M_0$ it follows that
 \begin{eqnarray*}
 \int_E |u_\delta(x)|^q\,\omega(x) \,dx&=&\int_E |v(x-\delta y_0)|^q\,\omega(x) \,dx=\int_{B_1(\delta y_0)} |v(x-\delta y_0)|^q\,\omega(x) \,dx\\&=&\int_{B_1} |v(y)|^q\,\omega(y+\delta y_0) \,dy=\delta^\tau\int_{B_1} |v(y)|^q\,\omega\(\frac{y}{\delta}+ y_0\) \,dy\\&\geq & \delta^\tau m_0\int_{B_1} |v(y)|^q \,dy.
 \end{eqnarray*}
In a similar way we obtain
  \begin{eqnarray*}
 	\int_E |\nabla u_\delta(x)|^p\,\sigma(x) \,dx&=&\int_E |\nabla v(x-\delta y_0)|^p\,\sigma(x) \,dx=\int_{B_1(\delta y_0)} |\nabla v(x-\delta y_0)|^p\,\sigma(x) \,dx\\&=&\int_{B_1} |\nabla v(y)|^p\,\sigma(y+\delta y_0) \,dy=\delta^\alpha\int_{B_1} |\nabla v(y)|^p\,\sigma\(\frac{y}{\delta}+ y_0\) \,dy\\&\leq & \delta^\alpha M_0\int_{B_1} |\nabla v(y)|^p \,dy.
 \end{eqnarray*}
 Accordingly, if we plug in the function $u_\delta$ into \eqref{WSI} with $\delta>1/r_0$, and use the last two estimates it follows that
$$\(\delta^\tau m_0\int_{B_1} |v(y)|^q \,dy\)^\frac{1}{q}\leq K_0\(\delta^\alpha M_0\int_{B_1} |\nabla v(y)|^p \,dy\)^\frac{1}{p}.$$
Letting $\delta\to \infty$, we obtain that $\dfrac{\tau}{q}\leq \dfrac{\alpha}{p}$. 
Now, using once again the dimensional balance condition \eqref{eq:dimensional balance}, we see that  the last inequality is equivalent to  
 \eqref{eq:p-range-bis}.

\subsection{Sobolev inequalities in the Heisenberg group}\label{sec:final comments and open questions}

In this part we consider the connection between weighted Sobolev inequalities in Euclidean cones and Sobolev inequalities in Heisenberg groups. Our original purpose was in fact to prove Sobolev inequalities in the Heisenberg group with sharp constants. 

For simplicity, we consider the first Heisenberg  group $\mathbb H^1$. Let us recall that $\mathbb H^1 = \R^3$ is endowed with its group operation given by
$$ (x,y,z) \ast (x',y',z') := \(x+x', y +y', z+z'+ \frac{1}{2}(xy'-yx')\).$$
In this setting one considers the {\it left invariant horizontal vector fields}  given by 
$X= \partial_x -\frac{1}{2} y \partial_z$ and $Y= \partial_y + \frac{1}{2} x \partial_z$ and the associated horizontal gradient  $\nabla_H u = X(u) X + Y(u) Y$.  For $p\in [1,4)$ we consider the Sobolev inequality
\begin{equation}\label{Sobolev-Heisenberg}
\biggl(\int_{\mathbb H^1}|u|^q\biggr)^{1/q}\leq C_p\biggl(\int_{\mathbb H^1}|\nabla_H u|^p\biggr)^{1/p},\ u\in C_0^\infty(\mathbb H^1),
\end{equation}
where $C_p>0$ and $q=\frac{4p}{4-p}$ is the Sobolev exponent given by scaling with Heisenberg dilations, where we have used the norm of the horizontal gradient for a function $u\in C_0^\infty(\mathbb H^1)$ given by 
$
|\nabla_Hu|=\sqrt{(Xu)^2+(Yu)^2}. %,\qquad Xu=u_x-\frac12 yu_z\,\quad Yu=u_y+\frac12x u_z.
$
In the following let us consider the class of functions $u$ that are  \textit{axially symmetric}: 
$$
u(x,y,z)=w(z,x^2+y^2).
$$
%
%$$u(x,y,z)=v(x_1,x_2)\ \  \textrm{with} \ \ x_1=z,\ \ x_2=x^2+y^2.$$
Then, by changing variables,  the Heisenberg Sobolev inequality (\ref{Sobolev-Heisenberg}) becomes equivalent to the Euclidean weighted Sobolev inequality 
\begin{equation} \label{eq:heisenberg-sobolev}
\biggl(
 \int_\R \int_0^\infty   w^q(x_1,x_2)dx_1 dx_2\biggr)^{1/q}\leq C_p\biggl( \int_\R  \int_0^\infty  \,|\nabla w|^p(x_1, x_2)  x_2^{p/2}dx_1 dx_2 \biggr)^{1/p}.
\end{equation} 
This problem fits well into the framework of this paper.  In fact, with our setup,  the open convex cone we are working with is $E=\mathbb R\times (0,\infty)$,  the weights being $\omega=1$ and $\sigma(x_1,x_2)=x_2^{p/2}$ for $(x_1,x_2)\in E$; accordingly, $\tau=0$ while $\alpha=p/2,$ and the fractional dimension is $n_a=4.$  Applying Theorem {\rm \ref{main-theorem-1}/(i) and Theorem \ref{main-theorem-2}/(i)} we obtain that \eqref{eq:heisenberg-sobolev} holds with constant
	\[  C_p=  \left\{ \begin{array}{lll}
	\displaystyle \frac{3p}{4-p} \, \inf_{\int_E v(y)\,dy=1, v\in C_0^\infty(\mathbb R^2), v\geq 0}
	\dfrac{\(\int_E v(y)\,|y|^{\frac{p}{p-1}}\,dy\)^{\frac{p-1}{p}}}{\int_E v(y)^{\frac{3}{4}}(y_2)^{\frac{1}{2}}\,dy}& {\rm if} &   p\in (1,4), \\
	% u\geq 0 &\mbox{in} &   \Omega;\\
\frac{5\pi^\frac{5}{4}}{2^\frac{13}{4}\Gamma^2\(\frac{3}{4}\)}	& {\rm if} &   p=1.\\
	\end{array}\right. \]

We do not know how to compute the explicit value of the constant $C_p$ for $p>1$. On the other hand, it is clear that this constant is not the sharp one for the inequality \eqref{eq:heisenberg-sobolev}, see Theorem \ref{main-theorem-1-equality}. It is in fact  still an open question to determine the sharp constant in both inequalities \eqref{Sobolev-Heisenberg} and  \eqref{eq:heisenberg-sobolev} for general values of $p$. 
When $p=2$, a sharp Sobolev inequality in the Heisenberg setting is due to Jerison and Lee \cite{JL} and it was proved also by a different method by Frank and Lieb \cite{FL}. Inequality (\ref{Sobolev-Heisenberg}) for $p=1$ is equivalent with Pansu's isoperimetric inequality; the Pansu's optimal constant is claimed to be 
$C_{\rm opt}=\frac{3^\frac{3}{4}}{4\sqrt{\pi}}<C_1$. There are several partial results related to Pansu's conjecture; we refer to the monograph of Capogna, Danielli, Pauls and Tyson \cite{CDPT} for a detailed account on this subject. 

\subsection{Open questions} We list here a few open problems related to results of this paper.  

\subsubsection{Sharp Sobolev inequalities with different weights.}
While the explicit computation of the constant $K_0$ in the statement of Theorem \ref{main-theorem-2} can be done by a direct calculation of the integrals in the expression of $K_0$, the computation of the value of $K_0$ in the statement of Theorem \ref{main-theorem-1}, even in case of simple weights, is a non-trivial matter. 

Motivated mainly by the Heisenberg setting from Section \ref{sec:final comments and open questions}, it would be interesting to further investigate whether the method of optimal transport can be used to obtain sharp constants in weighted Sobolev inequalities with different weights.

Another challenging question is to obtain Gagliardo-Nirenberg type inequalities with different weights. We notice that sharp Gagliardo-Nirenberg inequalities have been established by Del Pino and Dolbeault \cite{delPinoDolbeault-1, delPinoDolbeault-2} in the unweighted form, and by Lam \cite{Lam-1, Lam-2}  for identical homogeneous weights. 

 \subsubsection{Condition {\rm \ref{item:C_0}} and Bakry-\'Emery curvature-dimension condition.}
When $\omega= \sigma$, condition {\rm \ref{item:C_0}} is equivalent to the concavity of 
$\omega^{\frac{1}{\alpha}}$ that in turn characterizes the fact that the metric-measure space $(\R^n, d_E, \omega dx)$ satisfies the Bakry-\'Emery curvature-dimension condition $CD(0, n +\alpha)$ (see \cite[Remark 1.4]{Cabre-Ros-Oton-Serra} for details). 
Here, $d_E$ and $\omega dx$ are the usual Euclidean metric and the measure whose density with respect to the Lebesgue measure is $\omega$, respectively. It would be an interesting problem to find a geometric interpretation of condition {\rm \ref{item:C_0}} in terms of generalized curvature conditions of metric-measure spaces in the spirit of \cite{BGL, LV, Milman, S1, S2}. 

\subsubsection{On Muckenhoupt-Wheeden's weighted inequality}  To give a broader view, we close the paper mentioning earlier Sobolev inequalities for two weights in all space proved by  
Muckenhoupt and Wheeden \cite{MW} via fractional integration. They proved the following result: if $0<\gamma<n$, $1<p<n/\gamma$, and $\dfrac{1}{q}=\dfrac{1}{p}-\dfrac{\gamma}{n}$, then 
\begin{equation}\label{eq:weighted fractional integral inequality}
\|T_\gamma f \,V\|_{L^q(\mathbb R^n)}\leq C\, \|f\, V\|_{L^p(\mathbb R^n)}
\end{equation}
for all functions $f$ if and only if there exists $K>0$ such that 
\begin{equation}\label{eq:necessary and sufficient condition for V}
\(\fint_Q V(x)^q\,dx\)^{1/q}\(\fint_Q V(x)^{-p'}\,dx\)^{1/p'}\leq K
\end{equation}
for all cubes $Q\subset \R^n$. 
This condition is equivalent to $V^q$ belongs to the Muckenhoupt class $A_r$, with $r=1+q/p'$.
Here $T_\gamma$ stands for the fractional integral of order $\gamma$ given by 
\[
T_\gamma f(x)=\int_{\R^n}\dfrac{f(y)}{|x-y|^{n-\gamma}}\,dy.
\]
Using a representation formula of functions in terms of the fractional integral of order one of its derivatives, Muckenhoupt and Wheeden \cite[Theorem 9]{MW} deduced from \eqref{eq:weighted fractional integral inequality} when $\gamma=1$ a weighted Sobolev inequality of the form
\[
\|f\,V\|_{L^q(\mathbb R^n)}\leq C\,\(\|f\,V\|_{L^p(\mathbb R^n)}+\||\nabla f|\,V\|_{L^p(\mathbb R^n)}\).
\]

We notice that Muckenhoupt-Wheeden's condition and our condition {\rm \ref{item:C_1}} are rather  independent from each other. Indeed, if $V:\mathbb R^n\to (0,\infty)$ is any differentiable, homogeneous function of degree $\alpha\in \mathbb R$ and
$\omega(x)=V(x)^q$, $\sigma(x)=V(x)^p$ for every $x\in E=\R^n$, then $n_a=n$ and  $\sup_{x\in \mathbb R^n}\dfrac{\omega(x)^{1/q}}{\sigma(x)^{1/p}}=1.$ We observe that the pair $(\omega,\sigma)$ satisfies inequality \eqref{c-1-feltetel} in condition {\rm \ref{item:C_1}}
if and only if $V\equiv c$ for some $c>0$. Hence with this choice of the weights,  conditions  \eqref{eq:necessary and sufficient condition for V} and {\rm \ref{item:C_1}} are simultaneously satisfied in the whole $\mathbb R^n$ if and only if both weights are constant, i.e., $\omega(x)=c^q$, $\sigma(x)=c^p$, $x\in \mathbb R^n$.

Since our results are on cones, they are not in general comparable to these but nevertheless they raise the following methodological question:
is it possible to prove inequality \eqref{eq:weighted fractional integral inequality}, for example when $V=1$, by using optimal transport? 
This would be the analogue of the problem solved in \cite{CENV} for fractional integrals. In particular, it may give optimal constants and extremal functions for the fractional integral inequality as in Lieb \cite[Theorem 2.3]{L}.

\end{document}